\documentclass[preprint,11pt,authoryear]{elsarticle}
\usepackage[T1]{fontenc}
\usepackage{amsmath}
\usepackage{amssymb}
\usepackage{amsthm}
\usepackage{amsfonts}
\usepackage{mathtools}
\usepackage{tikz}
\usepackage{subcaption}

\usepackage[margin=2.5cm,a4paper]{geometry}
\usepackage{enumitem}
\usetikzlibrary{arrows.meta}
\usetikzlibrary{positioning}

\usepackage{enumitem}
\setlist{noitemsep,nolistsep}

\newcommand{\tluste}[1]{\mbox{\mathversion{bold}$ #1 $}}
\newcommand{\mna}[1]{{\mathcal{#1}}}

\newcommand{\onum}[1]{\mbox{$\overline{{#1}}$}} 
\newcommand{\ub}[1]{\onum{#1}}
\newcommand{\unum}[1]{\mbox{$\underline{{#1}}$}} 
\newcommand{\lb}[1]{\unum{#1}}
\newcommand{\inum}[1]{\mbox{$\tluste{#1}$}} 
\newcommand{\R}[0]{{\mathbb{R}}}

\newcommand{\N}[0]{{\mathbb{N}}}

\newcommand{\IR}[0]{{\mathbb{IR}}}
\newcommand{\MR}[1]{{#1}}

\newcommand{\mmid}[0]{\mid}	
\newcommand{\inter}[0]{\mathop{\mathrm{int}}}	%interior

\newcommand{\seznam}[1]{{\{1, \ldots, {#1}\}}}

\def\clqq{``}
\def\crqq{''}
\def\quo#1{\clqq{}#1\crqq{}}

\def\cw{c_w} %worst case value
    \tikzset{
        ordnode/.style={
            circle,
            minimum size=9mm,
            inner sep=0mm,
            draw,
            font=\small,
        },
    }

\usepackage[colorlinks]{hyperref}
\usepackage[capitalise,noabbrev]{cleveref}
\usepackage{microtype}

\newtheorem{theorem}{Theorem}

\newtheorem{proposition}[theorem]{Proposition}

\newtheorem{lemma}[theorem]{Lemma}
\newtheorem{corollary}[theorem]{Corollary}

\theoremstyle{definition}

\newtheorem{example}{Example}
\newtheorem*{motivation*}{Motivation}

\usepackage{setspace}
\AtBeginEnvironment{theorem}{\onehalfspacing}
\AtBeginEnvironment{assertion}{\onehalfspacing}
\AtBeginEnvironment{proposition}{\onehalfspacing}
\AtBeginEnvironment{observation}{\onehalfspacing}
\AtBeginEnvironment{lemma}{\onehalfspacing}
\AtBeginEnvironment{corollary}{\onehalfspacing}
\AtBeginEnvironment{conjecture}{\onehalfspacing}
\AtBeginEnvironment{open}{\onehalfspacing}
\AtBeginEnvironment{definition}{\onehalfspacing}
\AtBeginEnvironment{example}{\onehalfspacing}
\AtBeginEnvironment{remark}{\onehalfspacing}
\AtBeginEnvironment{proof}{\onehalfspacing}

\journal{arXiv}

\begin{document}
\onehalfspacing

\title{Minimum cost network flow with interval capacities:\\ The worst-case scenario}

\author[1,2]{Miroslav Rada\corref{cor1}}
\ead{miroslav.rada@vse.cz}

\author[3]{Milan Hladík}
\ead{hladik@kam.mff.cuni.cz}

\author[1,4]{Elif Radová Garajová}
\ead{elif.garajova@vse.cz}

\author[5]{Francesco Carrabs}
\ead{fcarrabs@unisa.it}

\author[5]{Raffaele Cerulli}
\ead{raffaele@unisa.it}

\author[5]{Ciriaco D'Ambrosio}
\ead{cdambrosio@unisa.it}

\affiliation[1]{organization={Prague University of Economics and Business, Faculty of Informatics and Statistics, Department of Econometrics},
addressline={nám. W. Churchilla 1938/4},
city={Prague},
postcode={130 67},
country={Czech Republic}}

\affiliation[2]{organization={Prague University of Economics and Business,  Faculty of Finance and Accounting, Department of Financial Accounting and Auditing},
addressline={nám. W. Churchilla 1938/4},
city={Prague},
postcode={130 67},
country={Czech Republic}}

\affiliation[3]{organization={Charles University, Faculty of Mathematics and Physics, Department of Applied Mathematics},
addressline={Malostranské nám.~2/25},
postcode={118 00},
city={Prague},
country={Czech Republic}}

\affiliation[4]{organization={Czech University of Life Sciences Prague, Faculty of Economics and Management, Department of Systems Engineering},
addressline={Kamýcká~129},
city={Prague},
postcode={165 00},
country={Czech Republic}}

\affiliation[5]{organization={University of Salerno, Department of Mathematics},
addressline={Via Giovanni Paolo II, 132},
city={Fisciano (SA)},
postcode={84084},
country={Italy}}

\cortext[cor1]{Corresponding author.}

\begin{abstract}
We study the problem of determining the worst optimal value and characterizing the corresponding worst-case scenarios in minimum cost network flow problems with interval uncertainty in arc capacities. In this setting, each capacity can take any value within its specified lower and upper bounds. We prove that computing the worst optimal value is a strongly NP-hard problem and remains NP-hard even when restricted to series–parallel graphs. Further, we propose a mixed-integer linear programming formulation that computes the exact worst optimal value, as well as a pseudopolynomial-time algorithm designed for the special case of series–parallel graphs.
We also examine the structural properties of the most extremal worst-case scenarios and show that the arcs whose capacities are not fixed at their interval bounds form a forest. This result establishes an upper bound on the number of such arcs, which we show to be tight by constructing a class of instances in which the bound is attained.
Finally, we investigate the more-for-less paradox in minimum cost network flow problems with interval capacities, which occurs in instances where increasing the required flow leads to a decrease in the worst-case optimal cost. We provide a general characterization of this phenomenon using augmenting paths and establish a stronger characterization for complete graphs. In addition, we discuss the properties of the cost matrices immune against the paradox and prove that deciding whether a given cost matrix has this property is a strongly co-NP-hard problem.
\end{abstract}


\maketitle

%%%%%%%%%%%%%%%%%%%%%%%%%%%%%%%%%%%%%%%%%%%%%%%%%%%%%%%%%%%%%%% 
\section{Introduction}

\paragraph{The minimum cost network flow}The minimum cost network flow problem is one of the fundamental models used in operations research and combinatorial optimization. The goal of the model is to find the most cost-efficient way to transport a given amount of flow from a source to a sink in a network while respecting the capacity constraints on the arcs. Pioneered by the seminal work of \citet{Ford:FlowsNetworks:1962}, the minimum cost network flow problem has since been the subject of extensive theoretical and algorithmic research. A comprehensive overview of its theoretical properties and efficient solution techniques can be found in the classical book by \citet{ahuja_network_1993}. Over the years, the model has found a wide range of practical applications, including logistics~\citep{Krile:ApplicationMinimumCost:2004}, healthcare~\citep{Slump:NetworkFlowApproach:1982}, or scheduling~\citep{Su:MinimumcostNetworkFlow:2013}.

\paragraph{Involving uncertainty}
However, in many practical optimization problems, the input data of the model are subject to uncertainty and may not be known with complete accuracy. To address this issue, various modeling frameworks have been developed to handle networks with uncertain parameters, aiming to identify solutions that perform well across a range of possible realizations. Prominent approaches to handling uncertain network flows include stochastic programming \citep{Estes:FacetsStochasticNetwork:2020}, robust optimization \citep{Mao:RobustDiscreteOptimization:2009} or fuzzy programming \citep{Ghatee:GeneralizedMinimalCost:2008}. 

\paragraph{Worst case over feasible scenarios}
This paper contributes to the line of research dealing with uncertainty in network flows by analyzing the worst-case behavior of the minimum cost network flow problem in which arc capacities are affected by interval uncertainty and may vary within prescribed lower and upper bounds while preserving feasibility of the requested flow through the network. 

\paragraph{Motivation 1: Incomparability of the feasible and infeasible scenarios}
The problem considered in this paper can be used to model and analyze adversarial uncertainty in network throughput. The restriction of the computation of the worst optimal value to feasible scenarios is suitable for applications in which infeasible scenarios correspond to a system collapse that is handled by mechanisms other than operational cost optimization (e.g. blackouts in energy networks, defaults in contractual obligations, or breakdowns of critical supply chains). In such contexts, cost optimization is only meaningful under continued operation.

Consider a freight transportation network with a given amount of flow to be transported from a source node to a sink at minimum cost. Each arc has an associated uncertain capacity, known to lie in a given interval, reflecting potential capacity degradation due to congestion, weather-related disruptions, or maintenance activities. For a given scenario determined by a realization of the capacities, the operator solves a standard minimum cost flow problem. While some capacity realizations correspond to infeasible scenarios representing network breakdowns in which transportation demand cannot be satisfied, in practice such situations are governed by emergency procedures (such as shipment cancellation, demand rationing, or a contractual force majeure clause), which are handled outside the minimum cost flow model. For this reason, we consider the worst finite optimal value, which represents the highest transportation cost over the scenarios in which the capacity degradations do not lead to a complete service failure and the network remains operational. This modeling framework provides a measure of network performance under interval uncertainty, distinguishing between cost escalation under operational stress and complete system collapse.

\paragraph{Motivation 2: Valuation of infeasibility} The worst finite optimal value also admits a natural interpretation in the context of contractual design for network problems under uncertainty in capacities. Consider a setting in which a service provider is contractually obligated to satisfy a given demand for transportation, with non-fulfillment resulting in a fixed penalty or fine. Capacity realizations that render the underlying network infeasible correspond to contract breach and trigger the penalty, whereas feasible realizations lead to an operational cost determined by a minimum cost flow solution. By evaluating the worst finite optimal value over all feasible scenarios, a contract designer can assess the maximum operational cost that may arise without incurring a breach, and thus determine a suitable value for the penalty associated with infeasibility. Conversely, from the service provider’s perspective, this value enables a comparison between the worst-case operational cost under feasible capacity degradations and the fine incurred under infeasibility, thereby informing the decision of whether entering the contract is economically justified.

\paragraph{Motivation 3: Network interdiction problem with feasibility-interested leader}
The network interdiction problem can be viewed as a zero-sum Stackelberg-type game. The leader chooses parameters of the network so as to restrict the follower’s feasible region, thereby forcing the follower’s optimal solution to be as unfavorable for him as possible, thus the most favorable for the leader. The framework of the network interdiction was thoroughly reviewed by \cite{Smith2013,Smith:SurveyNetworkInterdiction:2020}. In the formulations therein, the leader chooses the decision regardless of making the problem of the follower (in)feasible, thus not distinguishing between the worst case and the worst \emph{finite} case. This contrasts with our setup, where we allow only such decisions of leader for which the problem of the follower remains feasible. We provide two examples of network interdiction problems from the literature, where such a constraint occurs.

First, \cite{Hoppmann:FindingMaximumMinimum:2018}, apply the interdiction framework to German gas network, assuming the follower to be a network operator that wants to transport gas in the cheapest way possible, and the leader is a user of the network that sets supplies and demands of sources and sinks. Moreover, the contract between the network operator and the user requires the total supply and total demand to be balanced, thus ensuring feasibility of the corresponding network flow problem. In our setup, we have  interval capacities on arcs instead of interval supplies and demands (in turn: interval amount of flow). One can imagine that the cost of an arc means the extent of exhibition of media being transferred in the network to some (un)desirable effect (e.g. exposure to advertisement in transportation or communication networks or risk of police surveillance in drug trafficking networks). The leader wants to set capacities such that the total extent of exhibition is the largest possible, while the follower wants to minimize the exhibition. 

Second, \cite{wollmer:1970:algorithmsfortargetingstrikes} addressed the problem of targeting strikes in a lines-of-communication network (LOCP). The LOCP aims to strike parts of the network on the basis of the degree to which they hinder the combat force or the LOC user in achieving their goal. The capacities of the arcs depend on the number of strikes directed against them, and the goal of the LOC user is to achieve a minimum-cost flow from a source to a destination according to the capacities updated after the strikes. The setup in this paper can be modelled as the LOCP formulation, which is in fact more general in that it allows, for example, repairs of struck arcs or changes in arc costs induced by strikes. By contrast, our results cover a broader scope of properties of the problem and dive deeper in the problem, in particular, we discuss the \emph{more-for-less paradox} and its implications for the complexity of the problem.

\paragraph{Outcomes of our paper}We study the properties of the worst optimal value of a minimum cost network flow problem with interval capacities, analyze the theoretical complexity of its computation, and propose exact methods for determining it. We also characterize the structure of the scenarios in which the worst optimal value is attained. Furthermore, we address the more-for-less paradox, a counterintuitive phenomenon in which increasing the required flow may result in a lower worst-case optimal cost. Together, these results offer both novel theoretical insights into network flow problems under uncertainty and algorithmic tools for their solution. 

\subsection{Related results}

The minimum cost network flow problem under interval uncertainty has been previously addressed in the literature in different contexts.

\cite{Hoppmann:2021:MinCostFlow} studied the maximization and minimization variants of the minimum cost flow problem within a bilevel programming framework. The formulation allowed multiple source and sink nodes with interval-valued supplies and demands. In contrast to our model, arc capacities were treated as fixed real values (or possibly infinite) rather than intervals. Based on the special case of the interval transportation problem (ITP), the author proved that computing the worst optimal cost among all scenarios is NP-hard for the considered model.
The aim to generalize the ITP is also one of the motivations for this work. Some of the fundamental results for the ITP were successfully extended or reused. The exact algorithm from \cref{secExactComp} based on decomposition by complementary slackness uses essentially the same idea as the MILP-based formulation for ITP in \cite{Garajova:IntervalTransportationProblem:2023}. The concept of quasi-extreme scenarios (\cref{pro:upper:bound:n-1,pro:lower:bound:n-1}), which allows to compute the worst finite optimal value for ITP using a finite reduction \citep{CARRABS2021102492,Garajova:QuasiextremeReductionInterval:2024}. The convexity result (\cref{pro:convexity}) stems from its analogy in \cite{DAmbrosio:OptimalValueRange:2020}. 

\cite{HASHEMI20061200} examined a variant of the minimum cost flow problem with interval costs and fixed capacities. In the paper, they extended the classical combinatorial algorithms to the interval setting and solved the problem with respect to a complete ordering on intervals. 

\cite{Kasperski2008} considered the closely related problem of finding the minimum $s$--$t$ cut in a flow network with interval weights. The author surveyed and analyzed the min–max regret approach and investigated the algorithmic and complexity-theoretic properties of the model, proving NP-hardness of the problem for several classes of graphs. The monograph also discusses other discrete network optimization problems under interval uncertainty.

\cite{Singh-sensAnalysis} applied multiparametric sensitivity analysis to study the effects of perturbations in supply and demand, arc capacities, and costs in the minimum cost flow problem. They characterized the critical regions arising from simultaneous and independent parameter perturbations and also extended their analysis to multicommodity flows.

Minimum cost flow problems have also been investigated in the context of robust optimization. \cite{Busing:RobustMinimumCost:2022} analyzed the problem under consistent flow constraints with uncertain supply and demand. \cite{Chassein:ComplexityStrictRobust:2019} studied the robust integer minimum cost flow problem with uncertainty in the cost function. Both works also considered the class of series–parallel graphs and analyzed the computational complexity of the problem in this special case, which is also considered in this paper (\cref{propNpSerParGraph}).

Finally, \cite{MCFparadox} explored the more-for-less paradox in minimum cost network flow problems under perturbations. 
In their formulation, the initial arc capacities were given, and the central question was whether increasing certain capacities could lead to a reduction in the total minimum cost. In contrast, our analysis of the paradox only assumes the network structure and the arc costs as the input (\cref{thm:paradox}). Another related phenomenon was investigated by \cite{cenciarelli_polynomial-time_2019}, who derived a characterization of traffic networks exhibiting Braess' paradox. 

\subsection{Structure of the paper}
The remainder of the paper is organized as follows: 
\cref{sec:problem:form} formally introduces the minimum cost network flow problem with interval capacities, along with the relevant notation and definitions.
\cref{sec:properties} studies the fundamental properties of the model, focusing on the worst optimal value, the worst scenario and the set of feasible capacities.
\cref{sec:complexity} addresses the computational complexity of the problem, proving strong NP-hardness in the general case and NP-hardness for the special case of series-parallel graphs.
\cref{secExactComp} presents a mixed-integer linear programming formulation for computing the worst optimal value, as well as a pseudopolynomial algorithm for series-parallel graphs.
\cref{sec:worst:scen} examines the structural properties of the worst-case scenario and provides a bound on the number of arcs whose capacities are not set to their respective bounds.
\cref{sec:paradox} discusses the more-for-less paradox, including its characterization for general and complete graphs, and explores the properties of cost matrices that are immune to the paradox.
Finally, \cref{sec:conclusion} summarizes the main results, concludes the paper, and outlines possible directions for future research.

%%%%%%%%%%%%%%%%%%%%%%%%%%%%%%%%%%%%%%%%%%%%%%%%%%%%%%%%%%%%%%% 

\section{Problem Formulation}\label{sec:problem:form}

Let us first review the standard linear programming model of a minimum cost flow problem and introduce the necessary notation. Then, we incorporate interval-valued uncertainty affecting the arc capacity into the model and we formulate the problems addressed in this paper: computing the worst optimal value and characterizing the corresponding worst scenario.

Hereinafter, we denote by $\N$ the set of natural numbers including $0$, i.e. the set of all nonnegative integers. Similarly, the symbol $\R^+$ stands for the set of all nonnegative real numbers.

\paragraph{Minimum cost flow} 
Consider a directed graph $G=(V,E)$ with $n \coloneqq |V|$ nodes and $m \coloneqq |E|$ arcs. Further, let the following be given:
\begin{itemize}
    \item a source node $s\in V$,
    \item a sink node $t\in V$,
    \item arc capacities $u\colon E\to\R^+$; the value $u_e$ or $u_{ij}$ denotes the capacity of a given arc $e = (i,j) \in E$,
    \item cost function $c\colon E\to\R^+$; the value $c_e$ or $c_{ij}$ denotes the cost of sending a unit of flow along an arc $e = (i,j)\in E$,
    \item the amount of flow $f\in\N$ to be sent from $s$ to~$t$ (w.l.o.g assumed to be positive).
\end{itemize}
The goal of the minimum cost flow problem is to transport a flow of amount $f$ from the source $s$ to the sink $t$ with the minimal cost. Formally, the problem is stated as follows:
\begin{subequations}\label{lpFlow}
\begin{alignat}{3}
&\min \rlap{$\displaystyle\sum_{(i,j)\in E} c_{ij}x_{ij}$} \\ 
&\mbox{subject to} &
\label{lpFlow2}
 \sum_{j:(i,j)\in E} x_{ij} - \sum_{j:(j,i)\in E} x_{{ji}} &= 0,\qquad\qquad   &&\forall i\in V \setminus \{s,t\},\\
&&\label{lpFlow3} \sum_{j:(s,j)\in E} x_{sj} - \sum_{j:(j,s)\in E} x_{{js}} &= f,\\
&&\label{lpFlow4} \sum_{j:(t,j)\in E} x_{tj} - \sum_{j:(j,t)\in E} x_{{jt}} &= -f,\\
&&\label{lpFlow5} 0 \leq x_{ij} &\leq u_{ij},\quad &&\forall (i,j)\in E.
\end{alignat}
\end{subequations}

\paragraph{Interval capacities}
We consider a generalization of the network flow model, in which the arc capacities may be subject to interval uncertainty. 
Given two values $\lb{x}, \ub{x} \in \R$ satisfying $\lb{x} \le \ub{x}$, we define an interval $\inum{x} = [\lb{x}, \ub{x}]$ as the set $\{ x \in \R : \lb{x} \le x \le \ub{x} \},$ where $\lb{x}$ and $\ub{x}$ is the lower bound and the upper bound of the interval~$\inum{x}$, respectively.
The symbol $\IR$ is used to denote the set of all intervals,
and interval quantities are denoted by bold letters. 
Now, we can consider interval capacities on the arcs given by a mapping $\inum{u}\colon E\to\IR$ associated with a lower-bound mapping $\lb{u}\colon E\to\R^+$ and an upper-bound mapping $\ub{u}\colon E\to\R^+$ satisfying $\inum{u}_e = [\lb{u}_e, \ub{u}_e]$ for each arc $e \in E$. Without loss of generality assume that the lower and upper bounds determined by $\unum{u}$ and $\onum{u}$ are nonnegative integers.

\paragraph{Scenarios}
We use the term ``scenario'' to refer to a particular minimum cost flow problem instance or to a choice of arc capacities $u \in \inum{u}$ determining such an instance. Here, the choice $u \in \inum{u}$ means any function $u\colon E \to \R^+$ satisfying $u_e \in \inum{u}_e$ for each $e \in E$. Moreover, we denote by $G_u$ a particular instance of the graph with capacities $u\in\inum{u}$ and by $c(u)$ the corresponding minimum cost of the optimal flow.

\paragraph{The worst optimal value and the worst scenario}
The main problem addressed in this paper is to find a feasible scenario of the minimum cost flow problem with interval capacities, which yields the highest optimal value. 
Denote by $\mna{U}$ the set of feasible capacities with respect to a given flow $f$, i.e. the set
\begin{align*}
\mna{U}=\{u\in\inum{u}\mmid G_u\mbox{ admits a feasible flow of size $f$}\}.    
\end{align*}
\MR{Note that $\mna{U} \neq \emptyset$ if and only if the scenario with capacities $\onum{u}$ admits a flow of size $f$. Thus we can, without loss of generality, assume there exist feasible capacities for the given minimum cost flow problem, since this property can be efficiently verified in advance.}
Formally, we define the worst optimal value of the minimum cost flow problem with interval capacities $\inum{u}$ as the value
\begin{align*}
\cw\coloneqq\max\ c(u) \ \ \mbox{subject to}\ \ u\in\mna{U}.
\end{align*}
Moreover, we also refer to any scenario $u \in \mna{U}$, for which $c_w$ is achieved as the optimal value, as the worst scenario of the problem.

%%%%%%%%%%%%%%%%%%%%%%%%%%%%%%%%%%%%%%%%%%%%%%%%%%%%%%%%%%%%%%% 

\section{Basic properties of the model}\label{sec:properties}

We start our analysis of the minimum cost problem with interval capacities by studying some of the basic properties of the model related to the worst optimal value. Namely, we examine the properties of the optimal value function, the worst scenario and the set of all feasible scenarios.

\subsection{Set of feasible capacities}
In the definition of the worst optimal value $\cw$, we maximize the optimal value $c(u)$ subject to feasible capacities $u \in \mna{U}$. That is why, in this section, we focus on the properties of the set $\mna{U}$ of feasible capacities with respect to a given flow $f$.
We start by observing some fundamental geometric properties of the set~$\mna{U}$.

\begin{proposition}\label{pro:U:convex:polyhedron}
The set $\mna{U}$ is a convex polyhedral set.
\end{proposition}
\begin{proof}
Obviously, a given capacity $u\in\inum{u}$ is feasible if and only if there is a flow $x$ such that the pair $x$, $u$ satisfies the linear system \eqref{lpFlow2}--\eqref{lpFlow5}. Thus we have a polynomial representation of the set of feasible pairs by a linear system and the set $\mna{U}$ itself forms a projection into the subspace of variables~$u$. Thus, the statement of the proposition follows.
\end{proof}

Furthermore, if capacities are integral, the same can be said about vertices of $\mathcal{U}$. We remind that the requested flow is assumed to be an integer throughout the paper.
\begin{lemma}\label{pro:integrality:of:vertices:of:U}
If $\unum{u}$, $\onum{u}$ and the requested flow $f$ are integral, then vertices of $\mathcal{U}$ have integral coordinates.  
\end{lemma}
\begin{proof}
Consider $u \in \mathcal{U}$ such that $u_e $ is not an integer for at least one $e$. We show there is a shift $a \not = 0$ such that $u+a \in \mathcal{U}$, as well as $u-a \in \mathcal{U}$, which means that $u$ is in interior of a line segment in $\mathcal{U}$ and cannot be a vertex. When we will speak of $a$, we assume that all entries are zeroes except entries explicitly stated.

\newcommand{\nonint}{^{u\not \in \N}}
\newcommand{\xnonint}{^{x\not \in \N}}
Define $E\nonint \coloneqq \{ e \in E \mid u_e \text{ is not integer} \} $.  

Let $x$ be a feasible flow of size $f$ in $G_u$. If there is $e \in E\nonint$ such that $x_e < u_e$, it means that $x$ is feasible for $u + a$ and $u - a$ with $a_e > 0$ sufficiently small, meaning that $u$ is not a vertex of $\mathcal{U}$. Therefore, from now on, we will deal only with the case $x_e = u_e$ for all $e \in E\nonint$. 

Define $E\xnonint \coloneqq \{ e \in E \mid x_e \text{ is not integer} \}$ and $$V\xnonint \coloneqq \{ v \in V \mid (v, v') \in E\xnonint \lor (v',v) \in E\xnonint \text{ for some $v' \in V$}\}.$$ Note that there is an (undirected) cycle in graph $G\xnonint = (V\xnonint,E\xnonint)$. This is because each node in $V\xnonint$ has at least degree $2$: no node can have degree $0$ as we require each node to be incident with at least one edge, and no node can have degree $1$ because of the flow balance constraints -- in each node, the balance is integral ($f$, $-f$ or $0$), and if an edge transports nonintegral flow from or to a node, then there has to be at least one another edge that contributes to aligning the balance of the node to an integer value. In a graph with all node-degrees greater or equal to $2$, an (undirected) cycle has to exists. Denote the set of edges forming such a cycle by $C$. 

\newcommand{\forward}{^\text{f}}
\newcommand{\backward}{^\text{b}}
\newcommand{\cycleintegral}{^{C\land u \in \N}}
Now, we will show that we can both increase and decrease the nonintegral capacities on edges along the cycle $C$ (we will construct a suitable shift $a$). Fix the orientation of the cycle $C$ arbitrarily. Denote by $C\forward$ the set of forward-going edges in $C$ (i.e. edges whose orientation in $G$ is consistent with the chosen orientation), and $C\backward$ the set of backward-going edges in $C$.  Let $\varepsilon >0$ be a sufficiently small number.  Consider $a_e = \varepsilon$ for all $e \in C\forward \cap E\nonint$ and $a_e = -\varepsilon$ for all $e \in C\backward \cap E\nonint$. The capacities of all $e \in E\cycleintegral\coloneqq C \setminus E\nonint$ remain intact. We now show that both $u-a$ and $u+a$ are feasible scenarios. For scenario $u+a$, the flow $x'$ such that 
$$
x'_e = \begin{cases}
    x_e + \varepsilon & \text{if }e \in C\forward,\\
    x_e - \varepsilon & \text{if }e \in C\backward, \\
    x_e & \text{otherwise}.
\end{cases}$$
is feasible: the node flow-balance constraints \eqref{lpFlow2}--\eqref{lpFlow4} are met, the capacity constraints \eqref{lpFlow5} are met for edges from $E \setminus C$ (the flow and capacities were not changed) and edges from $ C \cap E\nonint$ (the capacities were changed the same way as the flow). For the remaining edges, i.e., for edges from the set $E\cycleintegral$, the following property holds: for each $e \in E\cycleintegral$, we have $0 < x_e < u_e$. Hence, $x'_e = x_e \pm \varepsilon $ is feasible. 

For scenario $u-a$, the flow witnessing feasibility is flow $x'$ with 
$x'_e = x_e - \varepsilon$ for $e \in C\forward$, $x'_e = x_e + \varepsilon$ for $e \in C\backward$ and $x'_e =x_e$ for $e \in E \setminus C$. The argumentation why this flow is feasible is the same as for scenario $u+a$.

To conclude, we proved that any feasible scenario $u$ with a nonintegral capacity on an edge is in the interior of line segment determined by two other feasible scenarios, meaning such a $u$ cannot be a vertex. 
\end{proof}

The question now is whether the set $\mna{U}$ admits an efficient representation by means of a system of linear inequalities. The answer is negative. \cref{prop:U:exp} shows that in the worst case, we need an exponential number of linear constraints to describe~$\mna{U}$.

\begin{theorem}\label{prop:U:exp}
 The set $\mna{U}$ cannot be described by a polynomial number of linear constraints. 
\end{theorem}

\begin{proof}
We construct an instance of the flow problem, for which an exponential number of inequalities is needed. Let $G$ be a complete graph with $n$ nodes, that is, $V=\seznam{n}$ and $E=\{(i,j)\in V^2\mmid i\not=j\}$. We set the flow $f=1$ and the interval capacities $\inum{u}_e=[0,2]$ for each $e \in E$. Now, capacity $u\in\inum{u}$ is feasible if and only if the size of each cut is at least~$f$. More formally, 
\begin{align}\label{ineqKnCut}
1=f \leq \sum_{i\in S,\,j\not\in S} u_{ij},\quad 
 \forall S: s\in S\subseteq V\setminus\{t\}. 
\end{align}
This system consists of $2^{n-2}$ linear inequalities, and we claim that none of them is redundant. Denote the inequality corresponding to a set $S$ by $ineq(S)$. If an inequality $ineq(S^*)$ is redundant, then it is a consequence of a nonnegative linear combination of some other inequalities $ineq(S_1),\dots,ineq(S_k)$. However, since these inequalities include variables not involved in $ineq(S^*)$ and because of positive coefficients by the variables, we cannot obtain such a nonnegative linear combination.

For a similar reason, none of the inequalities in \eqref{ineqKnCut} are redundant even if we include the lower bound constraints $u\geq\unum{u}=0$. 

Now, all vertices of the convex polyhedral set 
\[ \{u\mmid u\geq\unum{u}=0,\ u\mbox{ satisfies \eqref{ineqKnCut}}\} \]
lie in the hypercube $[0,1]^{|E|}$, since for the vertices some of the inequalities in \eqref{ineqKnCut} hold as equations. As a~result, the upper bound constraints $u_{ij}\leq\onum{u}_{ij}=2$ do not affect these vertices, and therefore none of the inequalities in \eqref{ineqKnCut} are redundant. 
\end{proof}

%%%
\subsection{Optimal value function}
Using the set of feasible capacities $\mna{U}$, we can equivalently define the worst optimal value as
$$
\cw=\max\ c(u)  \ \ \mbox{subject to}\ \  u\in\mna{U}.
$$
We know from \cref{pro:U:convex:polyhedron} that $\mna{U}$ is a convex polyhedral set. Here, we show that the optimal value function $c(u)$ is a convex function, which offers an insight into why the computation of $\cw$ is hard. Convexity of the function $c(u)$ can be inferred from the theory of parametric programming, but for the sake of completeness we include the proof.

\begin{proposition}\label{pro:convexity}
The optimal value function $c(u)$ is convex on~$\mna{U}$.    
\end{proposition}

\begin{proof}
For capacity $u\in\mna{U}$, we can replace the primal linear programming problem \eqref{lpFlow} by its dual problem, which has the form 
\begin{align*}
c(u) = \max\ u^Ty+d^Tz \ \ \mbox{subject to}\ \ (y,z)\in Z
\end{align*}
for a certain vector $d$ and a convex polyhedral set~$Z$. Now, for capacities $u_1,u_2\in\mna{U}$ and their convex combination $\lambda_1u_1+\lambda_2u_2$, we have
\begin{align*}
c(\lambda_1u_1+\lambda_2u_2)
&= \max_{(y,z)\in Z}\ (\lambda_1u_1+\lambda_2u_2)^Ty+d^Tz \\
&= \max_{(y,z)\in Z}\ \lambda_1(u_1^Ty+d^Tz)+\lambda_2(u_2^Ty+d^Tz) \\
&\leq \max_{(y,z)\in Z}\ \lambda_1(u_1^T
y+d^Tz)
 +\max_{(y,z)\in Z}\lambda_2(u_2^Ty+d^Tz) \\
&=\lambda_1c(u_1)+\lambda_2c(u_2),
\end{align*}
which proves convexity of the optimal value function $c(u)$ on~$\mna{U}$.
\end{proof}

\subsection{The worst scenario and optimal flows}

Let us now examine the properties of the scenarios, for which the worst optimal value $c_w$ is attained. We start by showcasing one instance in which the worst scenario can be easily characterized.

\begin{proposition}
    \label{pro:unum:u:feasible}
If the minimum cost flow problem is feasible for capacities $u\coloneqq\unum{u}$, then $\cw$ is attained for $u\coloneqq\unum{u}$.
\end{proposition}

\begin{proof}
For capacities $u\coloneqq\unum{u}$, the feasible set of linear program \eqref{lpFlow} is the smallest one among all scenarios, so the optimal value is maximal.
\end{proof}

A minimum cost flow problem with interval capacities can have multiple worst-case scenarios. Before we aim to find the worst scenario with a particular structure, let us discuss some useful properties of the optimal flows.
\cref{pro:minimal:capacities} shows that for any optimal flow within a worst-case scenario, there exists a worst-case scenario that has minimal capacities with respect to the flow.

\begin{lemma}\label{pro:minimal:capacities}
    Consider a scenario $u \in \mna{U}$. Let $x$ be an optimal flow in scenario $u$. Set $u' \coloneqq \max \{\unum{u}, x\}$, where the maximum is understood componentwise. Then $x$ is an optimal flow for scenario $u'$, too.
    \\
    Furthermore, if $u$ is a worst-case scenario, then $u'$ is worst-case scenario, too.
\end{lemma}
\begin{proof}
The feasible set of linear program \eqref{lpFlow} for scenario $u$ is a superset of the feasible set for the scenario~$u'$ (by definition of $u'$), hence $c(u) \le c(u')$ holds. Furthermore, $x$ is also a feasible solution of \eqref{lpFlow} for the scenario $u'$. Hence, $x$ witnesses $c(u) = c(u')$ and $x$ is optimal for $u'$.

Moreover, if $u$ is a worst-case scenario with $c(u) = \cw$, then also $c(u') = \cw$ holds since $c(u) = c(u')$.
\end{proof}

We can also observe that any optimal solution of a given scenario can be transformed into a solution, in which there is positive flow on arcs between any two nodes in at most one direction. Such a solution is called \emph{complementary}. \cref{prop:optSol:complem} formalizes and proves the result.

\begin{lemma}\label{prop:optSol:complem}
    For a scenario $u \in \mna{U}$, there exists an optimal solution of the associated minimum cost flow problem such that for each pair of arcs $(i,j), (j,i) \in E$ we have $x_{ij} \cdot x_{ji} = 0$. 
\end{lemma}
\begin{proof}
    Consider an optimal solution $x$ with positive flows on both $(i,j)$ and $(j,i)$ for some nodes $i, j \in V$. \MR{Construct a feasible solution $x'$ by subtracting flow of size $\varepsilon = \min\{ x_{ij}, x_{ji}\} > 0$ on both arcs $(i,j)$ and $(j,i)$, i.e. $x'_{ij} = x_{ij}-\varepsilon$ and $x'_{ji} = x_{ji} - \varepsilon$. 
    
    If $c_{ij} > 0$ or $c_{ji} > 0$, subtracting flow from both arcs strictly decreases the objective value and we obtain $c(x') < c(x)$, which contradicts optimality of the solution $x$.}
    \MR{Thus, $c_{ij} = c_{ji} = 0$ holds, since the arc costs are assumed to be nonnegative. As a consequence, we obtain $c(x') = c(x)$ and therefore $x'$ is an optimal solution of the problem, while also satisfying $x'_{ij} \cdot x'_{ji} = 0$. 
    
    This operation can be repeated for all pairs of nodes with a positive flow in both directions, yielding an optimal solution with the desired property.}
\end{proof}

\cref{pro:integer:capacities:means:integer:solution} states the integrality theorem that holds for the classical minimum cost flow problem. 

\begin{theorem}[{\citet[Theorem 9.10]{ahuja_network_1993}}]\label{pro:integer:capacities:means:integer:solution}
Given an integral scenario $u \in \mathcal{U}$ and integral requested flow $f$, $u$ admits an optimal flow $x$ such that $x$ is integral.
\end{theorem}

We now show that an analogous integrality theorem can be proven also for the worst scenario of the minimum cost flow problem with interval capacities.

\begin{theorem}\label{pro:integer:limits:means:integer:capacities:in:worst:scenario}
If capacities $\unum{u}$ and $\onum{u}$ and requested flow $f$ are integral, then there exists an integral scenario~$u$ with $c(u)= c_w$.
\end{theorem}
\begin{proof}
Recall that $c_w = \max c(u) \text{ s.t. }u \in \mathcal{U}$. Since $c(u)$ is a convex function according to \cref{pro:convexity} and $\mathcal{U}$ is a convex polyhedron (\cref{pro:U:convex:polyhedron}) with integral vertices (\cref{pro:integrality:of:vertices:of:U}), there exists an integral maximizer. 
\end{proof}

%%%%%%%%%%%%%%%%%%%%%%%%%%%%%%%%%%%%%%%%%%%%%%%%%%%%%%%%%%%%%%% 

\section{Computational complexity}\label{sec:complexity}
Let us now formally explore the computational complexity of the problem of computing the worst optimal value $c_w$. Note that our problem can be formulated as a worst case finite optimal value of a linear program with interval right-hand side \citep{Hla2018d}.
In general, this problem was proved to be NP-hard, however, due to the special structure of our flow problem, we have to inspect its computational complexity separately.

\begin{theorem}
    \label{pro:np:hard}
The problem of computing $\cw$ is strongly NP-hard.
\end{theorem}

\begin{proof}
    We use a reduction from the NP-hard problem of determining the longest path in an unweighted graph.

    To reduce the longest path to the problem of computing $\cw$, we consider all $\binom{n}{2}$ pairs $(s,t)$ of possible starting and ending nodes of the path. For each pair, we solve a flow problem instance on the graph with the required flow $f=1$, costs $c_{ij} = 1$ and interval capacities $\inum{u}_{ij}=[0,1]$ for all arcs. Then $\cw$ is attained for the longest path from $s$ to~$t$, and every integer optimal solution for $\cw$ yields the longest path from $s$ to $t$. 

    Hence, solving $O(n^2)$ problems of finding $\cw$ suffices to solve the NP-hard problem of determining the longest path in an unweighted graph.
\end{proof}

Since the above proof utilized a reduction from the NP-hard problem of finding the longest path (or a Hamiltonian path) on directed graphs, the problem of computing $\cw$ remains NP-hard even when restricted to those subclasses of graphs for which the longest (or Hamiltonian) path is hard to compute. This includes, e.g., directed planar graphs with indegree and outdegree at most two~\citep{Ple1979}.

\cref{propNpSerParGraph} presents another NP-hardness reduction from the knapsack problem. The reduction allows for deriving the hardness result even for the special class of series-parallel graphs (graphs with treewidth $2$). 
A series-parallel graph is a graph that can be constructed by a sequence of the following operations: 
\begin{enumerate}
    \item \emph{arc creation}: a new disjoint component of the graph is created, composed of a new source $s$, a new sink $t$ and the arc $(s,t)$,
    \item \emph{serial composition}: two disjoint components, one with source $s_1$ and sink $t_1$, one with source $s_2$ and sink $t_2$, are composed into one by merging $t_1$ with $s_2$, such that the resulting component has a source $s=s_1$ and a sink $t=t_2$,
    \item \emph{parallel composition}: two disjoint components, one with source $s_1$ and sink $t_1$, one with source $s_2$ and sink $t_2$, are composed into one by merging $s_1$ with $s_2$ and $t_1$ with $t_2$, such that the resulting component has a source $s = s_1=s_2$ and a sink $t=t_1=t_2$.
\end{enumerate}
Note that even though series-parallel graphs can have exponential numbers of paths from $s$ to $t$, the proposed reduction creates an instance of minimum cost flow, where only a linear number of paths is present. 

\begin{theorem}\label{propNpSerParGraph}
    The problem of computing $\cw$ is NP-hard for the class of series-parallel graphs and for graphs with $O(n)$ paths from $s$ to $t$. 
\end{theorem}
\begin{proof}
    We use a reduction from binary knapsack problem with integer data: Given a knapsack with capacity $b\in \N$ and a set of items $\{1, \dots, n'\}$ with weights $a \in \N^{n'}$ and values $c' \in \N^{n'}$, compute 
    \[ c_K \coloneqq \max {c'}^T y \text{ subject to } a^T y \le b,\ y \in \{0,1\}^{n'}.\]

    We construct an instance of the minimum cost flow problem with interval capacities on the series-parallel graph illustrated in \cref{fig:complex:SP} with a source node $s$, sink node $t$ and one node and three arcs per item in the knapsack problem. The numbers above arcs indicate the capacities, the numbers below arcs are the costs. The requested flow is set to $f = \min \{b, \sum_{i=1}^{n'} a_i\}$.
    
    \begin{figure}
    \centering
    \begin{tikzpicture}[->]
    \node[ordnode] (s) {$s$};
    \node[ordnode] (1) at (3, 2) {$1$};
    \node[ordnode] (2) at (3, -2) {$n'$};
    \node[ordnode] (t) at (6,0) {$t$};
    \node[rotate=90] at (1,0) {$\cdots$};
    \node[rotate=90] at (3,1) {$\cdots$};
    \node[rotate=90] at (3,-1) {$\cdots$};
    \node[rotate=90] at (3,0) {$\cdots$};

        \draw[] (s)--
                node[pos=0.5,sloped,yshift=0.25cm] {$[0,a_1]$} 
                node[pos=0.5,sloped,yshift=-0.25cm] {$0$}
            (1);
        \draw[] (2) to[bend left]
                node[pos=0.5,sloped,yshift=0.25cm] {$1$} 
                node[pos=0.5,sloped,yshift=-0.25cm] {$c'_{n'}$}
            (t);
            \draw[] (2) to[bend right]
                node[pos=0.5,sloped,yshift=0.25cm] {$a_{n'} -1$} 
                node[pos=0.5,sloped,yshift=-0.25cm] {$0$}
            (t);
        \draw[] (1) to[bend left]
                node[pos=0.5,sloped,yshift=0.25cm] {$1$} 
                node[pos=0.5,sloped,yshift=-0.25cm] {$c'_1$}
            (t);
        \draw[] (1) to[bend right]
                node[pos=0.5,sloped,yshift=0.25cm] {$a_1-1$} 
                node[pos=0.5,sloped,yshift=-0.25cm] {$0$}
            (t);
        \draw[] (s)--
                node[pos=0.5,sloped,yshift=0.25cm] {$[0,a_{n'}]$} 
                node[pos=0.5,sloped,yshift=-0.25cm] {$0$}
            (2);
\end{tikzpicture}
\caption{Reduction from the knapsack problem to the minimum cost flow with interval capacities. Numbers above the arcs are the capacities, numbers below the arcs are the costs.}\label{fig:complex:SP}
\end{figure}
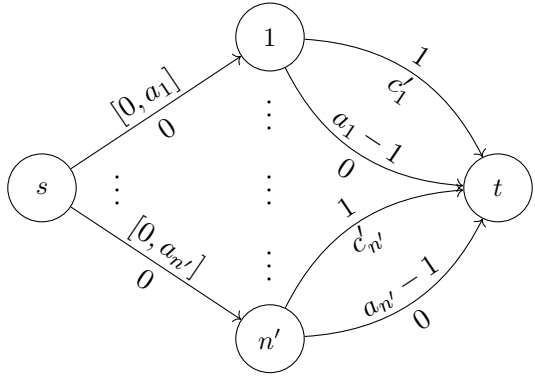 

Note that the constructed instance is a multigraph. This is without loss of generality, since parallel arcs can be subdivided by auxiliary nodes.
Note also that there are $2n' = O(n)$ paths from $s$ to $t$ and that the graph is series-parallel.

We claim that $c_K = \cw$ holds for the constructed instance. 

If $\sum_{i=1}^{n'} a_i \le b$, all items can be put into the knapsack and the knapsack problem is easy. In such a case, there is only one feasible scenario that allows to transfer flow $f$: we set the interval capacities to $u_{si} = a_i$ for every $i \in \{1, \dots, n'\}$. We have $\cw = \sum_{i=1}^{n'} c'_i = c_K$. 

In the rest of proof, we assume $\sum_{i=1}^{n'} a_i > b$, so we have $f=b$.

To prove $c_K \le \cw$, we consider an optimal solution $y^*$ of the knapsack problem with optimal value $c_K = {c'}^T y^*$. The total weight of items in this optimal solution is $w = a^T y^*$. We build a scenario $u$  with $c(u) = c_K$ as follows: we set $u_{si} = a_i$ for each $i$ such that $y^*_i = 1$. Since $y^*$ is an optimal solution of the knapsack problem, no item can be added to knapsack in this solution and we have $a_j > r \coloneqq b - w$ for each $j$ such that $y^*_j = 0$. We select whatever one of them, say $j'$ and set $u_{sj'} = r$. For all other $j$, we set $u_{sj} = 0$. Under this scenario, the minimum cost flow transports:
\begin{itemize}
\item $r$ units of flow through node $j'$, and since we have $r \le a_{j'} -1$, the costs are $0$,
\item $a_i$ units of flow through node $i$ for each $i$ such that $y^*_i = 1$, one of units is for the costs $c'_i$, other for zero costs.
\end{itemize}
Hence, the costs in this scenario is at least $c_K$ and we have $\cw \ge c_K$.

It remains to show that $c_K \ge \cw$. From \cref{pro:integer:capacities:means:integer:solution,pro:integer:limits:means:integer:capacities:in:worst:scenario}, we know that there is a scenario $u^*$ with $c(u^*) = \cw$ with optimal flow $x^*$, such that $x^*_e \in \N$ for all $e \in E$. 
We build solution $y'$ of the knapsack problem satisfying $c'^T y' = \cw$ and show that it is feasible. We simply set $y'_i = 1$ for every arc $(i,t)$ with nonzero costs that is used by flow $x^*$. Such a constructed $y'$ is feasible in the knapsack problem: $y'_i = 1$ iff the mentioned arc $(i,t)$ transports flow of size $1$, which occurs iff flow of size $a_i$ is transported through node $i$. Since $f = b$, we have $a^T y' \le b$. We have that $y'$ is a feasible solution of knapsack, so $c_K \ge c'^T y' = \cw$ and we are finished.
\end{proof}

\paragraph{Knapsack with rational data}
Note that the knapsack problem used in the proof of \cref{propNpSerParGraph} has integer coefficients. This allows to prove weak NP-hardness. However, with a more careful construction of the instance, we could also reduce the knapsack problem with rational coefficients, which is claimed to be strongly NP-hard \citep{wojtcak:2018:knapsackstronglynphard}.

%%%%%%%%%%%%%%%%%%%%%%%%%%%%%%%%%%%%%%%%%%%%%%%%%%%%%%%%%%%%%%%

\section{Exact computation of $\cw$}\label{secExactComp}

As we have proved in \cref{sec:complexity}, the problem of computing the worst optimal value of the minimum cost flow problem with interval capacities is NP-hard. Therefore, there is little hope to design an efficient polynomial-time method for finding the exact value $c_w$. 

In this section, we present a mixed-integer linear programming formulation for computing $c_w$ based on the theory of duality. The formulation allows us to utilize the available integer programming solvers and heuristics to tackle this challenging problem.

Furthermore, we also design a pseudopolynomial algorithm for computing $c_w$ on the subclass of series-parallel graphs, which has a time complexity polynomial in the number of nodes, arcs and the size of the flow. Recall that for this subclass, the problem was also proved to be NP-hard in \cref{propNpSerParGraph}.

\subsection{Mixed-integer linear programming formulation}
 Let us now derive a mixed-integer linear programming model for computing the worst optimal value~$c_w$ based on complementary slackness. An analogous approach was previously used for computing $c_w$ for the class of interval transportation problems~\citep{Garajova:IntervalTransportationProblem:2023}. Alternatively, we can also employ the general method by~\cite{Hla2018d} for computing the worst optimal value of an interval linear program based on a decomposition into basis stable regions.

Here, we first use duality and complementary slackness to model the problem as a linear program with additional indicator constraints (which are also supported by some integer programming solvers).

Let $\mathcal{S}$ denote the union of all network flows $x$ that are optimal for some setting of the capacities $u \in \mathcal{U}$. Then, the problem of computing the worst finite optimal value $\cw$ can be equivalently restated as finding the maximal value of the objective function over $\mathcal{S}$, i.e.
\[
\max \sum_{(i,j) \in E} c_{ij} x_{ij} \text{ subject to } x \in \mathcal{S}.
\]
To find a description of the optimal set $\mathcal{S}$, we can use duality in linear programming: the optimal solutions can be characterized by primal and dual feasibility and complementary slackness conditions (see also \citet[Chapter 9.4]{ahuja_network_1993} for details). This leads to the following formulation:
\begin{alignat*}{3}
    &\max \rlap{$\displaystyle\sum_{(i,j) \in E} c_{ij} x_{ij}$} \\
    &\text{ subject to } 
    &\text{constraints } &\eqref{lpFlow2}-\eqref{lpFlow5}, \\
    &&\pi_i - \pi_j - \alpha_{ij} &\le c_{ij}, &\forall (i,j) \in E,\\
    &&\alpha_{ij} &\ge 0, &\forall (i,j) \in E,\\
    &&c^\pi_{ij} &= c_{ij} - \pi_i + \pi_j, &\forall (i,j) \in E,\\
    &&\unum{u}_{ij} &\le u_{ij} \le \onum{u}_{ij}, &\forall (i,j) \in E,\\
    &&&\begin{cases}
    x_{ij} = 0 & \text{if } c^\pi_{ij} > 0,\\
    c^\pi_{ij} = 0 & \text{if } 0 < x_{ij} < u_{ij},\\
    x_{ij} = u_{ij} & \text{if } c^\pi_{ij} < 0.
    \end{cases}
\end{alignat*}
In this formulation, complementary slackness is implemented using the indicator constraints, which force the equations $x_{ij} = 0$ or $x_{ij} = u_{ij}$ depending on the sign of the variables $c^{\pi}_{ij}$. Note that the indicator constraint for $c^\pi_{ij} = 0$ can be omitted, since it is a consequence of the remaining two constraints: if $c^\pi_{ij} \neq 0$, then $x_{ij} \in \{0, u_{ij}\}$, therefore $x_{ij} \notin \{0, u_{ij}\}$ implies $c^\pi_{ij} = 0$.

Finally, we can reformulate the model as a mixed-integer program by restating the indicator constraints as big-M constraints. To this end, we introduce auxiliary binary variables $y_{ij}, z_{ij} \in \{0,1\}$ and replace each indicator constraint as follows:
\begin{align}
    c^\pi_{ij} > 0 &\implies x_{ij} = 0 &&  c^\pi_{ij} \le M y_{ij} \text{ and } x_{ij} \le M (1-y_{ij}) \\
    c^\pi_{ij} < 0 &\implies x_{ij} = u_{ij} && c^\pi_{ij} \ge -M z_{ij} \text{ and } u_{ij} - x_{ij} \le M(1-z_{ij})
\end{align}

A sufficiently large value of the constant $M$ can be defined as
\[ 
    M = \max \left\{ \max_{(i,j) \in E} \ub{u}_{ij}; \ \ 2\cdot\!\!\!\sum_{(i,j) \in E}\!\! c_{ij} + \max_{(i,j) \in E} c_{ij}    \right\}. 
\]
This setting allows the variables $c^\pi_{ij}$, $x_{ij}$, and the difference $u_{ij}-x_{ij}$ to attain a sufficiently large value. The value of the variable $x_{ij}$ and the expression $u_{ij}-x_{ij}$ can be bounded from above by the upper bound $\ub{u}_{ij}$ of the capacity on the corresponding arc.
The variable $c^\pi_{ij}$ is the sum of $-\pi_i$, $\pi_j$ and the cost of the arc $c_{ij}$, where the variables $\pi_i$ and $\pi_j$ are the shadow prices of the flow-balance constraint at nodes $i$ and $j$, respectively. Therefore, these variables represent the change in the objective value with respect to change in the right-hand side of the constraint and as such, their value can be bounded by the maximal cost of an (undirected) path or a cycle in the graph. Clearly, this cost can be bounded by the sum of the costs of all edges. Note that it is also possible to use separate values of $M$ for different variables to obtain tighter bounds.

\subsection{Pseudopolynomial algorithm for series-parallel graphs}
Let us now consider a special subclass of minimum cost network flow problems with interval capacities, in which the underlying network is a series-parallel graph. The task of computing the worst optimal value $c_w$ was also proved to be NP-hard on this class in \cref{propNpSerParGraph}. In this section, we propose an algorithm that is polynomial in the number of nodes $n$, the number of arcs $m$ and the size of flow~$f$, for the class of series-parallel graphs with integer values of capacities $\lb{u}, \ub{u}: E \rightarrow \N$ and a positive requested flow $f \in \N$.

Recall that a series-parallel graph is constructed using three types of operations: arc creation, parallel composition and serial composition (see \cref{sec:complexity}). During the construction, some components are being created and merged.
Observe that every arc creation operation creates $1$~new component: an arc with $2$ new nodes, while both types of composition operations merge $2$~existing components into $1$ more complex component, preserving the number of arcs and reducing the number of nodes by $1$ (for serial composition) or by $2$ (for parallel composition). Hence, every series-parallel graph with $m$ arcs has at most $2m$ nodes and can be constructed by at most $2m-1 = \mathcal{O}(m)$ operations.

\newcommand{\ur}{\unum{r}}
\newcommand{\hr}{\onum{r}}
\paragraph{Idea of the algorithm} 
Given an instance of the minimum cost flow problem with interval capacities $\inum{u}$ on a series-parallel graph $G$, we consider the finite sequence of operations $o_1, \ldots, o_q$ that constructs $G$. For a flow $x$ and capacities $u \in \inum{u}$, we say that a (directed) path $P$ in the graph is \emph{saturated} if $x_e = u_e$ holds for at least one $e \in P$.

For every component created during the construction, we will compute and store the following data:
\begin{itemize}
    \item the \emph{costs mapping} $d \colon \{0,1, \ldots, f\} \to \N \cup \{-\infty\}$ that stores the worst optimal costs at which a given flow can be sent through the component; if an amount of flow $f'$ is infeasible for the component, we assign $d(f') = -\infty$,
    \item the \emph{minimal restrictable flow} $\ur \in \N$, which is the minimal amount of flow that has to be sent through the component such that there exist some capacities $u \in \inum{u}$ for which all paths from the source to the sink are saturated, or, equivalently, the maximal amount of flow which can be transported through the component under all scenarios $u \in \inum{u}$, and, %transporting flow of size $\ur$ through the component is \MR{strongly feasible}, and,
    \item the \emph{maximal transportable flow} $\hr$ through the component over scenarios $u \in \inum{u}$; this information can actually be also computed from the costs mapping in time $O(f)$ as  $$\hr = \max\{f' \in \{0, 1, \dots, f\} \mmid d(f') \ge 0\}.$$
\end{itemize}  
    
After processing the entire construction sequence, the algorithm returns the maximal costs $d(f)$ attainable for the flow of size $f$ on the graph $G$, which correspond to the worst optimal value $c_w$. Note that thanks to the properties proved in \cref{pro:integer:capacities:means:integer:solution,pro:integer:limits:means:integer:capacities:in:worst:scenario}, it is sufficient to consider only integer flows during the computation.

To design the pseudopolynomial algorithm, we complete the following steps for each of the $\mathcal{O}(m)$ operations required to construct $G$:
\begin{itemize}
    \item Propose a method to compute the costs mapping $d$, the minimal restrictable flow $\ur$ (and possibly also the maximal flow $\hr$) for the components resulting from the three types of series-parallel graph operations.
    \item Prove that the computation for each type of operation depends only on $f$ and this dependence is polynomial.  
\end{itemize}
Below, we address these steps for each of the three types of operations.

\paragraph{Arc creation} During arc creation, a component with a single arc $e$ with a capacity $[\unum{u}_e, \onum{u}_e]$ and cost $c_e$ is created. Clearly, for this component the minimal restrictable flow $\ur$ can be set to $\lb{u}_e$, since this is the lowest capacity for which the only path (arc) from source to sink is saturated. Moreover, a flow of size $f' \le \ub{u}_e$ can be transported through the component, and for each $f' \in \{0, 1, \dots, \ub{u}_e\}$, the (maximal) costs are calculated simply as $d(f') = c_e \cdot f'$. Therefore, we have
$$
\ur = \unum{u}_e, \qquad \hr = \onum{u}_e, \qquad
d(f') = \begin{cases}
c_e \cdot f' & \text{if }f' \in \{ 0, 1, \ldots, \onum{u}_e \},\\
-\infty & \text{otherwise}.
\end{cases}
$$
The time-complexity of constructing the mapping is $\mathcal{O}(f)$.

\paragraph{Serial composition} 
Let $C_1$ and $C_2$ be the two original components used in the serial composition, with the minimal restrictable flows $\ur_1$, $\ur_2$, the maximal flows $\hr_1, \hr_2$ and costs mappings $d_1$, $d_2$, respectively, and let $C$ denote the newly composed component with the minimal restrictable flow $\ur$, the maximal flow $\hr$ and costs mapping $d$. In this operation, one of the sinks is merged with one of the sources, creating an articulation in the graph $C$, and all paths in $C$ are serial compositions of paths in the former components $C_1$, $C_2$. To send a flow of size $f'$ through the newly composed component, both $C_1$ and $C_2$ have to be able to transport $f'$. Therefore, the minimal restrictable flow (and also the maximal flow) depends on the more restrictive component and the maximal costs for transporting the flow are obtained as sum of the maximal costs within $C_1$ and $C_2$, i.e.
\[
\ur = \min\{\ur_1, \ur_2\}, \qquad \hr = \min\{\hr_1, \hr_2\}, \qquad
d(f') = d_1(f') + d_2(f') \quad \text{for }f' \in \{ 0,1, \ldots, f \}.
\]
The time complexity of constructing the mapping is $\mathcal{O}(f)$.

\paragraph{Parallel composition}
Let $C_1$ and $C_2$ be the two components used in the parallel composition, with {the} minimal restrictable flows $\ur_1$, $\ur_2$, {the} maximal flows $\hr_1, \hr_2$ and costs mappings $d_1$, $d_2$, respectively, and let $C$ denote the newly composed component with {the} minimal restrictable flow $\ur$, {the} maximal flow $\hr$ and costs mapping $d$. In this operation, the sources of $C_1$ and $C_2$ are merged to create source for $C$ and the sinks are merged to create the sink.

Note that all paths from $C_1$ and $C_2$ are available in $C$ independently, because no overlaps between the two former components are created in the composition  (other than the sources and sinks). Thus, the minimal restrictable flow and the maximal flow for $C$ can be computed as $$\ur = \ur_1 + \ur_2, \qquad \hr = \hr_1 + \hr_2. $$

The newly created component $C$ can generally offer more ways to transport a flow of a given amount. For a flow of size $f'$ to be transported in $C$, we distinguish three cases: 
\begin{enumerate}[label=\alph*)]
    \item Assume $\hr_1 + \hr_2 < f'$. In this case, the combined capacity of $C_1$ and $C_2$ is not sufficient to transport a flow of size $f'$ in $C$ and we can set $d(f') = -\infty$.
    
    \item \label{enu:pseudopoly:alg:unrestrictable} Assume $\ur_1 + \ur_2 > f'$. In this case, $f'$ can be transported through $C$, however, it cannot be transported as a sum of two restrictable flows. Therefore, any resulting flow of size $f'$ is not restrictable in $C$ and the requested flow $f'$ is feasible for all possible scenarios, meaning that we can compute $d(f')$ as the minimal cost flow for capacities $\unum{u}$ by any standard algorithm. There is a more efficient way, however: $d(f')$ can be expressed using $d_1, d_2, \ur_1$ and $\ur_2$ previously computed, namely 
    \begin{equation}\label{eq:parallel:b}
    d(f') = \min \left\{ d_1(f_1) + d_2( f' - f_1)\ \mid\ {f_1 \in [ f' - \ur_2, \ur_1] \cap [0,f']} \right\}.
    \end{equation}
    Let us justify this expression. Since the two composed parts of $C$ are independent, the flow $f'$ can be transported as sum of flows of size $f_1$ through $C_1$ and $f_2$ through $C_2$, and as long as $f_1 \le \ur_1$ and $f_2 \le \ur_2$, the flow $f'$ will be transported using the cheapest composition $f_1 + f_2$. 
    Furthermore, it holds that $f_1 \in [0,f']$. The sums of flows using $f_1 > \ur_1$ or $f_2 > \ur_2$ would require increasing capacity on at least one arc, and hence they cannot result in worse costs than the expression \eqref{eq:parallel:b}, see also \cref{pro:unum:u:feasible}.

    \item \label{enu:pseudopoly:alg:restrictable} Assume $\ur_1 + \ur_2 \le f'$  and $\hr_1 + \hr_2 \ge f'$. This is the most interesting case: in general, both restrictable and unrestrictable amounts of flows can contribute to $f'$. We actually show that it suffices to consider only sums of two restrictable flows.
    
    We claim that \begin{equation}\label{eq:pseudopoly:df:restrictable:case}d(f') = \max \left\{d_1(f_1) + d_2(f' - f_1)\ \mid\ {f_1 \in [\max\{\ur_1, f' - \hr_2\} , \min\{f'-\ur_2, \hr_1\}]}\right\}.\end{equation}

    Clearly, the maximum in \eqref{eq:pseudopoly:df:restrictable:case} is sought over such decompositions of $f'$ to $f_1, f_2$ such that $f_1 \in [\ur_1,\hr_1]$ and $f_1 = f'-f_2 \in [f' - \hr_2, f' - \ur_2] $ implying $f_2 \in [\ur_2, \hr_2]$. So, the computed $d(f')$ can be  obtained by two restrictable flows and the value $d(f')$ can be enforced by some scenario. 

    \newcommand{\con}{^{\text{con}}}
    
    It remains to show that there is no scenario with worse total costs. For contradiction, assume there is a scenario $u\con$ transporting $f'$ with costs $c\con > d(f')$. The flow can be decomposed to flows $f\con_1$ and $f\con_2$ through $C_1$ and $C_2$. Either $f\con_1 < \ur_1$, or $f\con_2 < \ur_2$, since all other cases are examined by \eqref{eq:pseudopoly:df:restrictable:case}. Without loss of generality assume that $f\con_1 < \ur_1$. In scenario $u\con$,  flow of size $\ur_1$ can be transported through $C_1$ with costs $c_1 = d_1(\ur_1)$ (as we can assume all capacities in $C_1$ are set to the their lower limits) and flow $f'-\ur_1$ can be transported through $C_2$ with costs $c_2 \le d_2(f'-\ur_1)$ (flow of this size is feasible as $f' -\ur_1 < f' - f\con_1 = f\con_2$). But flow with costs $c_1 + c_2$ was not chosen as optimal solution of min cost flow with capacities $u\con$, which means that $c_1 + c_2 \ge c\con$. Furthermore, we have $d(f') \ge d_1(\ur_1) + d_2(f'-\ur_1)$, since if $f\con_1 < \ur_1$, then $\ur_1 > f' - f\con_2 \ge f' - \hr_2$ and $f_1 = \ur_1$ is one of the values examined in maximization in \eqref{eq:pseudopoly:df:restrictable:case}.  Putting it together, we finally obtain the contradiction: $$c_1 + c_2 \ge c\con > d(f') \ge d_1(\ur_1) + d_2(f'-\ur_1) \ge c_1 + c_2.$$

\end{enumerate}
In total, cases b) and c) seek for minimum or maximum over $O(f)$ values of flow $f_1$ through $C_1$. This is repeated for $O(f)$ values of flow $f'$. In total, parallel composition can be performed in time $O(f^2)$. 

We conclude the description of the algorithm with the easy proposition about its complexity.
\begin{proposition}\label{pro:complexity:pseudpolynomial:algo}
The algorithm described above computes $\cw$ in time $O(m f^2)$, where $m$ is the number of arc-creation operations.  
\end{proposition}

%%%%%%%%%%%%%%%%%%%%%%%%%%%%%%%%%%%%%%%%%%%%%%%%%%%%%%%%%%%%%%% 

\section{The most extremal worst scenario}\label{sec:worst:scen}

In \cref{pro:unum:u:feasible} we proved that if the lower bound capacities are feasible, then they correspond to the worst case scenario. However, in many instances the feasibility assumption may not be satisfied and the minimal capacities may be too low to transport the required flow. 
Therefore, we now explore the characterization of the worst scenario with the most extremal setting of the interval capacities. 

We say that a capacity setting $u_{ij}\in\inum{u}_{ij}$ is \emph{interior} if $u_{ij}\in\inter\inum{u}_{ij}$, that is, $u_{ij}\not\in\{\unum{u}_{ij},\onum{u}_{ij}\}$. \cref{pro:forest,pro:upper:bound:n-1} address the question what is the minimum number of interior values that can be obtained for a scenario corresponding to~$\cw$.

\begin{theorem}\label{pro:forest}
    The worst case value $\cw$ is attained for a scenario $u \in \mna{U}$ such that the arcs with interior capacities form a forest.
\end{theorem}
\begin{proof}
Consider a worst case scenario $u \in \mna{U}$. We will show that we can modify the scenario $u$ such that its arcs with interior values have the desired property.
Firstly, we can without loss of generality consider the worst case scenario $u$ that is minimal with respect to its complementary optimal solution $x$ (see \cref{pro:minimal:capacities} and \cref{prop:optSol:complem}).

The following procedure iteratively updates the scenario $u$ and the flow $x$ in order to reduce the number of arcs with interior capacities:
\newcommand{\inte}{\text{int}}
\begin{enumerate}
    \item Let $G^{\inte}_u$ be the subgraph of $G_u = (V, E)$ containing only arcs with interior capacities, i.e.
        \[ G^{\inte}_u = (V, E^{\inte} \coloneqq \{(i,j) \in E\ \mmid\ (i,j)\text{ has interior capacity $u_{ij}$}\}). \]
    \item Find any cycle (regardless of the orientation of arcs) in the graph $G^{\inte}_u$. If none exists, end the procedure -- the arcs with interior values form a forest. Otherwise, let the cycle consist of a sequence of nodes $i_1,i_2,\dots,i_{K}=i_1 \in V$.
    \item Let $\mna{K}\coloneqq\{1,\dots,K-1\}$. Compute the desirable change of the flow $\delta$ on the cycle: 
    \begin{multline*}
\delta\coloneqq\min\big\{
 \min\{\onum{u}_{i_k,i_{k+1}}-{x}_{i_k,i_{k+1}}\mmid (i_k,i_{k+1})\in E^{\inte},\, k\in\mna{K}\},\\
 \min\{x_{i_{k+1},i_{k}} - \unum{u}_{i_{k+1},i_k}\mmid (i_{k+1},i_k)\in E^{\inte},\, k\in\mna{K}\} \big\}.
\end{multline*}
\item For each $k\in\mna{K}$, we set:
\begin{align*}
    x_{i_k,i_{k+1}}&=x_{i_k,i_{k+1}}+\delta, \quad\qquad \text{if } (i_k,i_{k+1})\in E^{\inte},\\
    x_{i_{k+1},i_k}&=x_{i_{k+1},i_k}-\delta, \quad\qquad \text{if } (i_{k+1},i_{k})\in E^{\inte}.
\end{align*}
\item For each arc $e$ on the cycle found in Step 2, set $u_e = x_e$. 
\item Go to Step 1.
\end{enumerate}

It is clear from Step 2 that once this procedure terminates, it yields a new scenario $u^*$ such that the arcs with interior capacities form no cycles -- hence its graph is a forest. 
Since there is at least one capacity set to a non-interior value in each iteration (by definition of $\delta$) and only capacities with interior values are modified, the procedure terminates after a finite number of iterations. To finish the proof, we show that the new scenario is also a valid worst case scenario for the given problem.

First, we can see from the definition of $\delta$ in Step 3 and from the modification of $x$ in Step 4 that the new capacities defined in Step 5 are valid with respect to the given upper and lower bounds. Therefore, we have $u^* \in \inum{u}$.

Further, we claim that if the optimal value of the original scenario $u \in \mna{U}$ is the worst optimal value~$c(u) = c_w$, then the optimal value of the modified scenario $u^*$ is also $c_w$. Namely, we show that the complementary optimal solution $x$ of the worst case scenario $u$ (with the cost $c_w$) is modified into a flow~$x^*$, which is a complementary optimal solution of the scenario $u^*$ with the same cost.

\MR{Consider a cycle $C$ and fix the orientation of the cycle. The cost of the oriented cycle is the sum of the costs of its arcs, where a forward arc (i.e. an arc of $G$ consistent with the chosen orientation) contributes $c_e$ to the total cost, while a backward arc (inconsistent with the orientation) contributes $-c_e$.} Note that the cycles found in Step 2 are of zero cost. Otherwise, \MR{we could choose an orientation of the cycle (in Step 2) such that} applying the modifications in Steps~4 and 5 would lead to a scenario $u'$ with a higher optimal cost $c(u') > c(u)$, which contradicts the assumption that the original scenario $u$ is a worst case scenario.

Denote by $x'$ and $u'$ the updated solution and capacities obtained in Steps 4 and 5, respectively. In each iteration, the solution $x'$ is a feasible flow with respect to $u'$:
\begin{enumerate}[label=\alph*)]
    \item After changing the flow on the cycle in Step 4, the solution $x'$ still forms a valid flow (neglecting the capacities), since the flow conservation constraints remain satisfied.
    \item After updating the capacities in Step 5, the flow $x'$ from Step 4 is a feasible solution for the modified scenario $u'$, since all capacity (and flow) constraints are respected.
\end{enumerate}

Finally, after Steps 4 and 5, the flow $x'$ is an optimal solution for the scenario $u'$. This can be justified by the characterization of optimality of a minimum cost flow using the residual graph, which only depends on those arcs in the original graph where the current flow can be increased (with $x_e < u_e$) or decreased (with $x_e > 0$). It is known that a given flow is optimal if and only if a residual graph contains no negative cycles (see \citet[Thm 9.1]{ahuja_network_1993} for details). Here, after changing the flow and capacities on the cycle, no new arcs emerge in the residual graph. Hence, if the original flow $x$ was an optimal solution for $u$ and no negative cycles were present in the residual graph, no new negative cycle can arise after the update, so the new flow $x'$ is optimal for the new scenario $u'$.
Moreover, the flow $x'$ is a complementary optimal flow for the scenario $u'$, which is minimal with respect to $x'$.
\end{proof}

\begin{corollary}\label{pro:upper:bound:n-1}
The worst case value $\cw$ is attained at a scenario $u\in\mna{U}$ such that at most $n-1$ capacity values of $u$ are interior.
\end{corollary}
\begin{proof}
    Follows directly from \cref{pro:forest}, since a forest on $n$ nodes has at most $n-1$ arcs.
\end{proof}

Furthermore, we can prove that the bound on the number of interior capacities obtained in \cref{pro:upper:bound:n-1} is tight. \cref{exa:lower:bound:n-1} illustrates a network with $n$ nodes, for which there are $n-1$ interior values in each worst case scenario. The result is formally proved in \cref{pro:lower:bound:n-1}.

\begin{example}\label{exa:lower:bound:n-1}
    Consider the graph with $n$ nodes $V = \{1, \dots, n\}$ and $2n-2$ arcs depicted in \cref{fig:n-1:edges:interior}. The numbers above the arcs correspond to the capacities, the numbers below the arcs are the costs.
    Namely, the set of arcs $e \in E$ is defined as follows:
    \begin{align*}
        (i, i+1) &\in E,  & c_{e} &= 2, &    u_{e} &= n-i,        && \text{for } i \in \{1, \dots, n-4\}, \\
        (i, i+1) &\in E,  & c_{e} &= 2, &    u_{e} &= n-i-1,      && \text{for } i \in \{n-3,n-2\}\\
        (n-1, n) &\in E,  & c_{e} &= 2, &      u_{e} &= 2,\\
        (i, n) &\in E,    & c_{e} &= 1, &        \inum{u}_{e} &= [0, 2], && \text{for } i \in \{1, \dots, n-2\}, \\
        (n-3,n-1) &\in E, & c_{e} &= 1, &    \inum{u}_{e} &= [0, 2].
    \end{align*}
     The aim is to transport flow of size $f=n$ from the source node $s=1$ to the sink $t=n$. 

        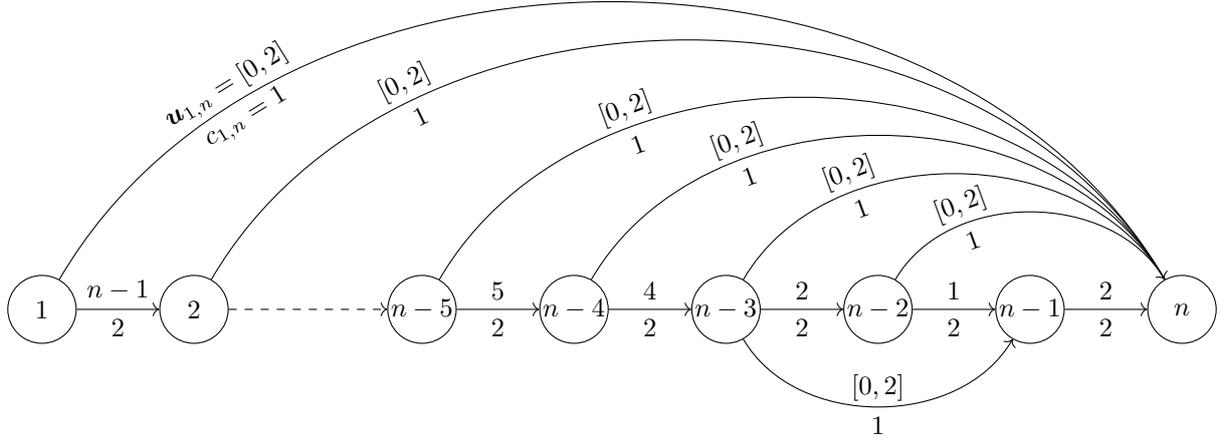
\begin{figure}[t]
        \begin{center}
    \begin{tikzpicture}[every node/.style={font=\small}]
        \usetikzlibrary{calc}
        \begin{scope}[every node/.style=ordnode]
            \foreach \n/\shift in {1/-3.5,2/-1.5,n-5/1.5,n-4/3.5,n-3/5.5,n-2/7.5,n-1/9.5,n/11.5}
            \node[label={[label position=center]$\n$}] (\n) at (\shift,0) {};
        \end{scope}
            %\node[align=center] (dots) at ($(2)!0.5!(n-3)$) {\LARGE$\cdots$};
        \draw[->] (1)--
                node[pos=0.5,sloped,yshift=0.25cm] {$n-1$} 
                node[pos=0.5,sloped,yshift=-0.25cm] {$2$}
            (2);
        \draw[->,dashed] (2)--
            (n-5);
        %\draw[,dashed] (2)--(dots);
        %\draw[,dashed] (dots)--(n-3);
        \draw[->] (n-5)--
                node[pos=0.5,sloped,yshift=0.25cm] {$5$} 
                node[pos=0.5,sloped,yshift=-0.25cm] {$2$}
            (n-4);
        \draw[->] (n-4)--
                node[pos=0.5,sloped,yshift=0.25cm] {$4$} 
                node[pos=0.5,sloped,yshift=-0.25cm] {$2$}
            (n-3);
        \draw[->] (n-3)--
                node[pos=0.5,sloped,yshift=0.25cm] {$2$} 
                node[pos=0.5,sloped,yshift=-0.25cm] {$2$}
            (n-2);
        \draw[->] (n-2)--
                node[pos=0.5,sloped,yshift=0.25cm] {$1$} 
                node[pos=0.5,sloped,yshift=-0.25cm] {$2$}
            (n-1);
        \draw[->] (n-1)--
                node[pos=0.5,sloped,yshift=0.25cm] {$2$} 
                node[pos=0.5,sloped,yshift=-0.25cm] {$2$}
            (n);
        \begin{scope}[out=60,in=120]
            \draw[] (1) to 
                node[pos=0.2,sloped,yshift=0.25cm] {$\inum{u}_{1,n}=[0,2]$} 
                node[pos=0.2,sloped,yshift=-0.25cm] {$c_{1,n}=1$}
                (n) ;
        \draw[] (2) to 
                node[pos=0.25,sloped,yshift=0.25cm] {$[0,2]$} 
                node[pos=0.25,sloped,yshift=-0.25cm] {$1$}
            (n);
        \draw[] (n-5) to
                node[pos=0.3,sloped,yshift=0.25cm] {$[0,2]$} 
                node[pos=0.3,sloped,yshift=-0.25cm] {$1$}
            (n);
        \draw[->] (n-4) to
                node[pos=0.3,sloped,yshift=0.25cm] {$[0,2]$} 
                node[pos=0.3,sloped,yshift=-0.25cm] {$1$}
            (n);
        \draw[] (n-3) to
                node[pos=0.3,sloped,yshift=0.25cm] {$[0,2]$} 
                node[pos=0.3,sloped,yshift=-0.25cm] {$1$}
            (n);
        \draw[] (n-2) to
                node[pos=0.3,sloped,yshift=0.25cm] {$[0,2]$} 
                node[pos=0.3,sloped,yshift=-0.25cm] {$1$}
            (n);
        %\draw[] (n-1) to (n);
        \end{scope}
        \draw[->,out=-60,in=-120] (n-3) to
                node[pos=0.5,sloped,yshift=0.25cm] {$[0,2]$} 
                node[pos=0.5,sloped,yshift=-0.25cm] {$1$}
            (n-1);
    \end{tikzpicture}
\end{center}
    \caption{Minimum cost flow problem with interval capacities shown in \cref{exa:lower:bound:n-1}.}
    \label{fig:n-1:edges:interior}
\end{figure}
\end{example}
     
Note that in \cref{exa:lower:bound:n-1} there are exactly $n-1$ arcs with an interval capacity. \cref{pro:lower:bound:n-1} proves that in the worst case scenario, the capacity on all of these arcs has to be set to an interior value. Therefore, these arcs form the tree constructed in the proof of \cref{pro:forest}. 

\begin{proposition}\label{pro:lower:bound:n-1}
        For the minimum cost flow problem with interval capacities presented in \cref{exa:lower:bound:n-1}, the worst optimal value $c_w$ is attained only for the scenario $u \in \mna{U}$ with $u_e = 1$ for every arc $e \in E$ with an interval capacity. 
\end{proposition}

\begin{proof}
        There are two types of arcs: \emph{curved arcs} with interval capacities and \emph{straight arcs} with crisp capacities. Consider the scenario $u \in \mna{U}$ from the statement, in which $u_e = 1$ for all curved arcs.
        
        Clearly, the problem is feasible for scenario $u$, since flow of size $n$ can be transported through the network.

        We prove the following auxiliary claim:
        Given a node $i \in \{1,\ldots, n-4\}$, assume that a flow of size $n-i+1$ has to be transported through $i$. If $c(u') = c_w$ for some feasible scenario $u' \in \mna{U}$, then $u'_{i,n} = 1$.
        
        First, assume for contradiction that $u'_{i,n} < 1$. Then $u'$ is not a feasible scenario, since the total capacity of arcs leaving $i$ is $n-i + u'_{i,n} < n-i+1$, so the required flow cannot be transported.
        
        Second, assume that $u'_{i,n} > 1$. Note that the arc $(i,n)$ is the least expensive path from $i$ to $n$ with costs $c_{i,n}=1$. \MR{
        Then, there exists another feasible scenario $u'' \in \inum{u}$ with $u''_{i,n} = 1$}, in which the flow $u'_{i,n}-1$ has to be transported using the straight arc $(i,i+1)$ with the cost at least $c_{i,i+1}=2 > c_{i,n}$. However, the scenario $u''$ then satisfies $c(u') < c(u'') \le c_w$.

        Hence, we can assume that $u_{i,n} = 1$ for each $i\in \{1,\ldots, n-4\}$ and that a flow of size $4$ has to be transported from node $n-3$ to $n$. There are $4$ available paths, each of them transporting exactly one unit of flow under the scenario $u$:
        \begin{enumerate}[label={Path \arabic*:},leftmargin=*,noitemsep]
            \item $(n-3) \rightarrow n$ with the cost $1$,
            \item $(n-3) \rightarrow (n-2) \rightarrow n$ with the cost $3$,
            \item $(n-3) \rightarrow (n-2) \rightarrow (n-1) \rightarrow n$ with the cost $6$, and,
            \item $(n-3) \rightarrow (n-1) \rightarrow n$ with the cost $3$.
        \end{enumerate}
        Note that the capacity of Path 3 is crisp and cannot be increased. To make the flow more expensive, some units of flow have to be transported using Paths 2 and 4. However, this is not possible: increasing the flow by $\delta$ on these paths increases the costs by $2\cdot3\cdot\delta$; however, it simultaneously decreases the cost of the flow on Path 3 by $6\delta$, and also, on Path 1 by $\delta$. Hence, any changes in capacities forcing transportation of the flow along Paths 2 and 4 lead to a decrease of cost by $\delta$. We can observe that no changes in capacities lead to the cost $c(u)$ being increased or even remaining the same. Thus, $u$ is the unique worst scenario.
    \end{proof}

An analogous result was previously proved for the transportation problem with interval supplies and demands \cite{CARRABS2021102492}, for which the number of interior values of the worst case scenario can be limited to at most one. This result may lead to the idea that the important parameter influencing the number of interior values is the distance from $s$ to $t$. However, a similar example as above can also be constructed when the distance between $s$ and $t$ is much smaller, so the distance does not directly influence the number of interior values, in general.

%%%%%%%%%%%%%%%%%%%%%%%%%%%%%%%%%%%%%%%%%%%%%%%%%%%%%%%%%%%%%%% 
\section{The more-for-less paradox}\label{sec:paradox}

\tikzset{colorLine/.style={very thick, color={blue}, text={black}}}

In classical transportation problems, a counterintuitive phenomenon may arise in certain instances, where an increase in the total volume of goods transported actually leads to a reduction in the overall transportation cost. This phenomenon, referred to as the more-for-less paradox, has been the subject of study and discussion in the existing literature. A similar situation can also be observed in the more general model of minimum cost flow problems with interval capacities: in some instances, increasing the requested flow $f$ may, surprisingly, lead to a decrease of the worst optimal cost $\cw$.

Specifically, given an instance of the minimum cost flow problem with interval capacities, we say that the \emph{more-for-less paradox} occurs if there exist requested flow amounts $f$ and $f'$ with $f < f'$ such that their respective worst optimal values $\cw, \cw'$ satisfy $\cw' < \cw$. Therefore, requesting a higher amount of flow leads to a decrease in the worst optimal value of the corresponding minimum cost flow problem. Notice that, without loss of generality, we can restrict the examination of the occurrence of the paradox to the case $f'=f+1$, since we can simply increase the requested flow one-by-one until the paradox occurs.

\begin{example}\label{exa:paradox:simple}
Consider a network flow instance on the graph illustrated in \cref{fig:ex:paradox}. The numbers above the arcs denote the interval capacities and the numbers below the arcs are costs.

For the requested flow $f = 1$, we obtain the worst optimal value $\cw = 12$ for the scenario $u_{s2} = u_{1t} = 0$ (Figure~\ref{fig:ex:paradox:2a}). However, for the increased flow $f=2$, we have $\cw = 4$, since the capacities have to be set to $u_{s2} = u_{1t} = 1$ in order to make the problem feasible (Figure~\ref{fig:ex:paradox:2b}). The increased capacities allow to send $1$ unit of flow along each of ``cheap'' paths $s$--$2$--$t$ and $s$--$1$--$t$, instead of the expensive path $s$--$1$--$2$--$t$, which was used for $f=1$.

\begin{figure}[ht]\centering
    \begin{tikzpicture}[->]
    \node[ordnode] (s) {$s$};
    \node[ordnode] (1) at (1.5, -1.5) {$1$};
    \node[ordnode] (2) at (3, 1.5) {$2$};
    \node[ordnode] (t) at (4.5,0) {$t$};

        \draw[] (s)--
                node[pos=0.5,sloped,yshift=0.25cm] {$1$} 
                node[pos=0.5,sloped,yshift=-0.25cm] {$1$}
            (1);
        \draw[] (1)--
            node[pos=0.5,sloped,yshift=0.25cm] {$1$} 
                node[pos=0.5,sloped,yshift=-0.25cm] {$10$}
            (2);
        \draw[] (2)--
                node[pos=0.5,sloped,yshift=0.25cm] {$1$} 
                node[pos=0.5,sloped,yshift=-0.25cm] {$1$}
            (t);
        \draw[] (1)--
                node[pos=0.5,sloped,yshift=0.25cm] {$[0,1]$} 
                node[pos=0.5,sloped,yshift=-0.25cm] {$1$}
            (t);
        \draw[] (s)--
                node[pos=0.5,sloped,yshift=0.25cm] {$[0,1]$} 
                node[pos=0.5,sloped,yshift=-0.25cm] {$1$}
            (2);
\end{tikzpicture}
\caption{An instance of the minimum cost flow problem with interval capacities admitting the more-for-less paradox (see \cref{exa:paradox:simple}).}\label{fig:ex:paradox}
\end{figure}
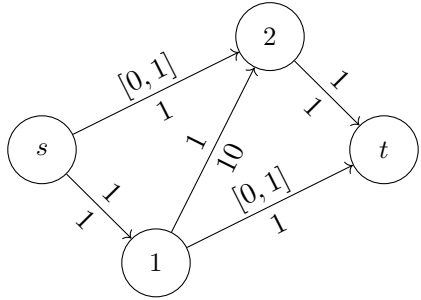 

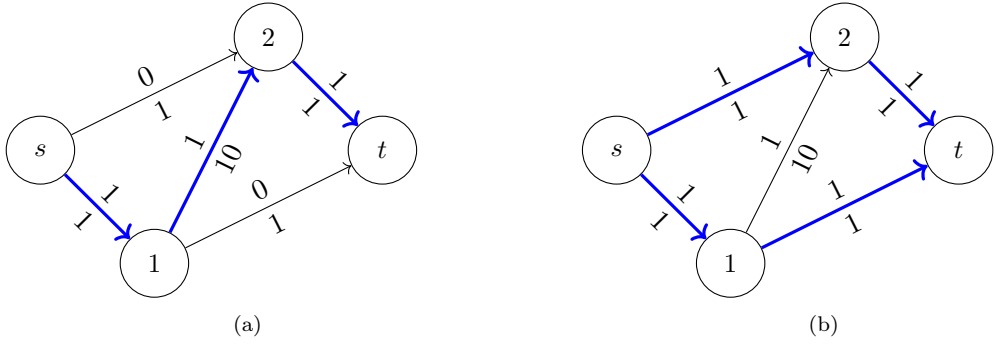
\begin{figure}[ht]
\centering
\begin{subfigure}{.4\textwidth}	
 \begin{tikzpicture}[->]
    \node[ordnode] (s) {$s$};
    \node[ordnode] (1) at (1.5, -1.5) {$1$};
    \node[ordnode] (2) at (3, 1.5) {$2$};
    \node[ordnode] (t) at (4.5,0) {$t$};

        \draw[colorLine] (s)--
                node[pos=0.5,sloped,yshift=0.25cm] {$1$} 
                node[pos=0.5,sloped,yshift=-0.25cm] {$1$}
            (1);
        \draw[colorLine] (1)--
            node[pos=0.5,sloped,yshift=0.25cm] {$1$} 
                node[pos=0.5,sloped,yshift=-0.25cm] {$10$}
            (2);
        \draw[colorLine] (2)--
                node[pos=0.5,sloped,yshift=0.25cm] {$1$} 
                node[pos=0.5,sloped,yshift=-0.25cm] {$1$}
            (t);
        \draw[] (1)--
                node[pos=0.5,sloped,yshift=0.25cm] {$0$} 
                node[pos=0.5,sloped,yshift=-0.25cm] {$1$}
            (t);
        \draw[] (s)--
                node[pos=0.5,sloped,yshift=0.25cm] {$0$} 
                node[pos=0.5,sloped,yshift=-0.25cm] {$1$}
            (2);
\end{tikzpicture}
  \caption{}
  \label{fig:ex:paradox:2a}
\end{subfigure}%
\hspace{30pt}
\begin{subfigure}{.4\textwidth}	
 \begin{tikzpicture}[->]
    \node[ordnode] (s) {$s$};
    \node[ordnode] (1) at (1.5, -1.5) {$1$};
    \node[ordnode] (2) at (3, 1.5) {$2$};
    \node[ordnode] (t) at (4.5,0) {$t$};

        \draw[colorLine] (s)--
                node[pos=0.5,sloped,yshift=0.25cm] {$1$} 
                node[pos=0.5,sloped,yshift=-0.25cm] {$1$}
            (1);
        \draw[] (1)--
            node[pos=0.5,sloped,yshift=0.25cm] {$1$} 
                node[pos=0.5,sloped,yshift=-0.25cm] {$10$}
            (2);
        \draw[colorLine] (2)--
                node[pos=0.5,sloped,yshift=0.25cm] {$1$} 
                node[pos=0.5,sloped,yshift=-0.25cm] {$1$}
            (t);
        \draw[colorLine] (1)--
                node[pos=0.5,sloped,yshift=0.25cm] {$1$} 
                node[pos=0.5,sloped,yshift=-0.25cm] {$1$}
            (t);
        \draw[colorLine] (s)--
                node[pos=0.5,sloped,yshift=0.25cm] {$1$} 
                node[pos=0.5,sloped,yshift=-0.25cm] {$1$}
            (2);
\end{tikzpicture}
  \caption{}
  \label{fig:ex:paradox:2b}
\end{subfigure}
\caption{The more-for-less paradox in \cref{exa:paradox:simple}. For the requested flow $f=1$ we have $c_w = 12$ \textit{(a)} , while for $f=2$  we obtain $c_w = 4$ \textit{(b)}.}\label{fig:ex:paradox:2}
\end{figure}
\end{example}

We will show that the more-for-less paradox is related to the existence of augmenting paths for the flow. Formally, an \emph{augmenting path} \MR{for a flow $x$ }in the graph is a sequence of arcs corresponding to an undirected path from the source $s$ to the sink $t$ with $x_{ij} < u_{ij}$ for each forward arc $(i,j)$ on the path and $x_{ij} > 0$ for each backward arc on the path.\MR{ Note that the cost of an augmenting path is calculated as follows: a forward arc $e$ contributes $c_e$ to the total cost, whereas a backward arc $e$ contributes $-c_e$.}

\MR{The crucial property of the cost function of the instance in \cref{exa:paradox:simple} is that if we consider flow along the path $s$--$1$--$2$--$t$, there is an augmenting path $s$--$2$--$1$--$t$ (with respect to the scenario $\onum{u}$ shown in \cref{fig:ex:paradox:2b}).} However, \cref{exa:paradox:complex} illustrates that the more-for-less paradox can also occur in networks with a more complex structure, showing that there can be multiple sequences of reversed arcs on an augmenting path.

\begin{example}\label{exa:paradox:complex}
Consider the graph depicted in \cref{fig:ex:paradox:complex}. Again, the numbers above arcs are capacities and the numbers below arcs are costs.

For the requested flow $f = 1$, we have the worst optimal value $\cw = 29$ for the scenario $u_{s2} = u_{14} = u_{3t}= 0$ (Figure~\ref{fig:ex:paradox:complex:2a}). However, for the increased flow $f=2$ we have $\cw = 27$, since the capacities have to be set to $u_{s2} = u_{14} = u_{3t}= 1$ to make the requested flow feasible (Figure~\ref{fig:ex:paradox:complex:2b}). Now, the increased capacities allow to send $1$ unit of flow along each of ``cheaper'' paths $s$--$1$--$4$--$t$ and $s$--$2$--$3$--$t$, instead of the expensive path $s$--$1$--$2$--$3$--$4$--$t$ used for $f=1$.

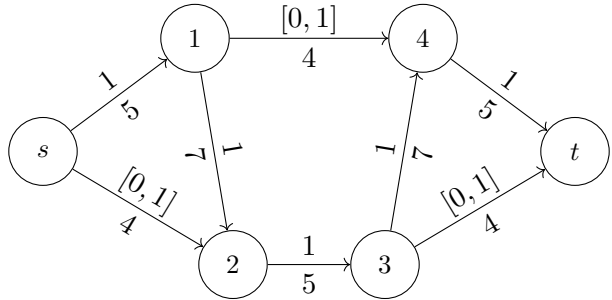
\begin{figure}[ht]
\begin{center}
    \begin{tikzpicture}[->]
    \node[ordnode] (s) {$s$};
    \node[ordnode] (1) at (2, 1.5) {$1$};
    \node[ordnode] (2) at (2.5, -1.5) {$2$};
    \node[ordnode] (3) at (4.5, -1.5) {$3$};
    \node[ordnode] (4) at (5, 1.5) {$4$};
    \node[ordnode] (t) at (7,0) {$t$};

        \draw[] (s)--
                node[pos=0.5,sloped,yshift=0.25cm] {$1$} 
                node[pos=0.5,sloped,yshift=-0.25cm] {$5$}
            (1);
            \draw[] (s)--
                node[pos=0.5,sloped,yshift=0.25cm] {$[0,1]$} 
                node[pos=0.5,sloped,yshift=-0.25cm] {$4$}
            (2);
        \draw[] (1)--
            node[pos=0.5,sloped,yshift=0.25cm] {$1$} 
                node[pos=0.5,sloped,yshift=-0.25cm] {$7$}
            (2);
        \draw[] (3)--
                node[pos=0.5,sloped,yshift=0.25cm] {$1$} 
                node[pos=0.5,sloped,yshift=-0.25cm] {$7$}
            (4);
        \draw[] (2)--
                node[pos=0.5,sloped,yshift=0.25cm] {$1$} 
                node[pos=0.5,sloped,yshift=-0.25cm] {$5$}
            (3);
            \draw[] (1)--
                node[pos=0.5,sloped,yshift=0.25cm] {$[0,1]$} 
                node[pos=0.5,sloped,yshift=-0.25cm] {$4$}
            (4);
        \draw[] (3)--
                node[pos=0.5,sloped,yshift=0.25cm] {$[0,1]$} 
                node[pos=0.5,sloped,yshift=-0.25cm] {$4$}
            (t);
            \draw[] (4)--
                node[pos=0.5,sloped,yshift=0.25cm] {$1$} 
                node[pos=0.5,sloped,yshift=-0.25cm] {$5$}
            (t);
\end{tikzpicture}
\end{center}
\caption{A more complex instance of the minimum cost flow problem with interval capacities admitting the more-for-less paradox (see \cref{exa:paradox:complex}).}\label{fig:ex:paradox:complex}
\end{figure}

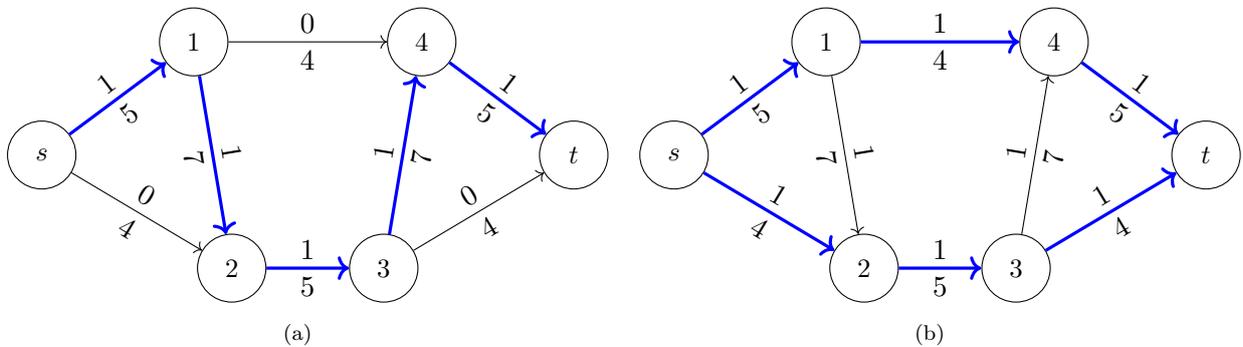
\begin{figure}[ht]
\begin{center}
\begin{subfigure}{.48\textwidth}	
    \begin{tikzpicture}[->]
    \node[ordnode] (s) {$s$};
    \node[ordnode] (1) at (2, 1.5) {$1$};
    \node[ordnode] (2) at (2.5, -1.5) {$2$};
    \node[ordnode] (3) at (4.5, -1.5) {$3$};
    \node[ordnode] (4) at (5, 1.5) {$4$};
    \node[ordnode] (t) at (7,0) {$t$};

        \draw[colorLine] (s)--
                node[pos=0.5,sloped,yshift=0.25cm] {$1$} 
                node[pos=0.5,sloped,yshift=-0.25cm] {$5$}
            (1);
            \draw[] (s)--
                node[pos=0.5,sloped,yshift=0.25cm] {$0$} 
                node[pos=0.5,sloped,yshift=-0.25cm] {$4$}
            (2);
        \draw[colorLine] (1)--
            node[pos=0.5,sloped,yshift=0.25cm] {$1$} 
                node[pos=0.5,sloped,yshift=-0.25cm] {$7$}
            (2);
        \draw[colorLine] (3)--
                node[pos=0.5,sloped,yshift=0.25cm] {$1$} 
                node[pos=0.5,sloped,yshift=-0.25cm] {$7$}
            (4);
        \draw[colorLine] (2)--
                node[pos=0.5,sloped,yshift=0.25cm] {$1$} 
                node[pos=0.5,sloped,yshift=-0.25cm] {$5$}
            (3);
            \draw[] (1)--
                node[pos=0.5,sloped,yshift=0.25cm] {$0$} 
                node[pos=0.5,sloped,yshift=-0.25cm] {$4$}
            (4);
        \draw[] (3)--
                node[pos=0.5,sloped,yshift=0.25cm] {$0$} 
                node[pos=0.5,sloped,yshift=-0.25cm] {$4$}
            (t);
            \draw[colorLine] (4)--
                node[pos=0.5,sloped,yshift=0.25cm] {$1$} 
                node[pos=0.5,sloped,yshift=-0.25cm] {$5$}
            (t);
\end{tikzpicture}
\caption{}
\label{fig:ex:paradox:complex:2a}
\end{subfigure}
\hfill
\begin{subfigure}{.48\textwidth}	
\begin{tikzpicture}[->]
    \node[ordnode] (s) {$s$};
    \node[ordnode] (1) at (2, 1.5) {$1$};
    \node[ordnode] (2) at (2.5, -1.5) {$2$};
    \node[ordnode] (3) at (4.5, -1.5) {$3$};
    \node[ordnode] (4) at (5, 1.5) {$4$};
    \node[ordnode] (t) at (7,0) {$t$};

        \draw[colorLine] (s)--
                node[pos=0.5,sloped,yshift=0.25cm] {$1$} 
                node[pos=0.5,sloped,yshift=-0.25cm] {$5$}
            (1);
            \draw[colorLine] (s)--
                node[pos=0.5,sloped,yshift=0.25cm] {$1$} 
                node[pos=0.5,sloped,yshift=-0.25cm] {$4$}
            (2);
        \draw[] (1)--
            node[pos=0.5,sloped,yshift=0.25cm] {$1$} 
                node[pos=0.5,sloped,yshift=-0.25cm] {$7$}
            (2);
        \draw[] (3)--
                node[pos=0.5,sloped,yshift=0.25cm] {$1$} 
                node[pos=0.5,sloped,yshift=-0.25cm] {$7$}
            (4);
        \draw[colorLine] (2)--
                node[pos=0.5,sloped,yshift=0.25cm] {$1$} 
                node[pos=0.5,sloped,yshift=-0.25cm] {$5$}
            (3);
            \draw[colorLine] (1)--
                node[pos=0.5,sloped,yshift=0.25cm] {$1$} 
                node[pos=0.5,sloped,yshift=-0.25cm] {$4$}
            (4);
        \draw[colorLine] (3)--
                node[pos=0.5,sloped,yshift=0.25cm] {$1$} 
                node[pos=0.5,sloped,yshift=-0.25cm] {$4$}
            (t);
            \draw[colorLine] (4)--
                node[pos=0.5,sloped,yshift=0.25cm] {$1$} 
                node[pos=0.5,sloped,yshift=-0.25cm] {$5$}
            (t);
\end{tikzpicture}
\caption{}
\label{fig:ex:paradox:complex:2b}
\end{subfigure}
\end{center}
\caption{The more-for-less paradox in \cref{exa:paradox:complex}. For the requested flow $f=1$ we have $c_w = 29$ \textit{(a)}, while for $f=2$ we obtain $c_w = 27$ \textit{(b)}.}\label{fig:ex:paradox:complex:2}
\end{figure}
\end{example}

\subsection{Characterization of the paradox}
Let us now examine the structure of the network flow problems admitting the more-for-less paradox. \cref{propParadoxNec} shows a necessary condition for the paradox to occur based on augmenting paths.

\begin{proposition}[Necessary condition for the paradox]\label{propParadoxNec}
Let $G = (V, E)$ be a directed graph with a source~$s \in V$, sink $t \in V$ and a nonnegative cost function $c\colon E \to \R^+$. If the more-for-less paradox occurs for some interval capacities $\inum{u}\colon E \rightarrow \IR$ and a~flow size~$f$, then there exists an $s$--$t$ flow $x^f$ of size~$f$ feasible for a scenario $u \in \inum{u}$ and an augmenting path for $x^f$ (with respect to $u$) with a negative cost.
\end{proposition}

\begin{proof}
Let $u^f$ be the capacities corresponding to the worst case scenario for flow of size~$f$, and analogously $u^{f+1}$ for a flow of size $f+1$. Let $x^f$ and $x^{f+1}$ be the corresponding optimal flows. Define $u_e^*\coloneqq \max\{u_e^f,u_e^{f+1}\}$ for each arc $e \in E$.

Define the difference of flows $x^*\coloneqq x^{f+1}-x^{f}$, which corresponds to a flow of size~1. Since $c(u^f)>c(u^{f+1})$, the total cost of $x^*$ is negative. If \MR{the arcs with a nonzero flow in $x^*$ correspond} to a path, then the proof is complete, since \MR{they form} an augmenting path for $x^f$ under capacity $u^* \in \inum{u}$. Otherwise, $x^*$ consists of a path $P$ and several cycles. 

Either the path $P$ has a negative cost and thus forms an augmenting path from the statement, or there is at least one cycle~$C$ with a negative cost. 
However, the latter case contradicts the assumption that $x^{f+1}$ is the minimum cost flow of size~$f+1$, since we can decrease the cost of the flow by involving the cycle~$C$.
\end{proof}

On the other hand, \cref{propParadoxSuf} provides a sufficient condition for the paradox to occur in a minimum cost flow problem with interval capacities.

\begin{proposition}[Sufficient condition for the paradox]\label{propParadoxSuf}
Let $G = (V, E)$ be a directed graph with a source~$s \in V$, sink $t \in V$ and a nonnegative cost function $c\colon E \to \R^+$.
If there \MR{exist capacities $u \colon E \to \R^+$,} an $s$--$t$ flow $x^f$ of size~$f$ \MR{feasible for $u$} and an augmenting path for $x^f$ with a negative cost, then there exist interval capacities admitting the more-for-less paradox.
\end{proposition}

\begin{proof}
Let $x^f$ be a flow of size $f$ and the total cost $c^f$, and denote by $P$ the augmenting $s$--$t$ path with a negative cost. Denote by $x^{f+1}$ the flow of size $f+1$ created from $x^f$ by sending in additional unit of flow along the path~$P$. Further, set the interval capacities on all arcs to $\inum{u}_e\coloneqq [0,\max \{ x_e^f,x_e^{f+1} \}]$. Denote by $c^f_w$ and $c^{f+1}_w$ the worst optimal costs corresponding to the interval capacities $\inum{u}$ and flow sizes $f$ and $f+1$, respectively.

First, observe that $c^f\leq c^f_w$ holds, since $x^f$ is an optimal flow for capacities $u\coloneqq x^f$. Second, since $P$ is an augmenting path with negative cost, then $c^{f+1}<c^f\leq c^f_w$. 

If $x^{f+1}$ is the worst case optimal solution, then $c^{f+1}_w=c^{f+1}<c^f_w$ holds, finishing the proof. Otherwise, there is a worst-case solution $x^*$ with cost $c^{f+1}_w>c^{f+1}$. In this case, we can obtain $x^*$ from $x^{f+1}$ by a finite sequence of worsening cycles, i.e. cycles where increasing the flow by one unit along it also increases the total cost. 

Consider such a worsening cycle~$C$. The cycle~$C$ uses some arcs from the augmenting path $P$ and some arcs outside $P$. On the arcs $e \notin P$ the flow is decreased, since otherwise it would exceed the upper bound capacity $\ub{u}_e = \max \{ x_e^f,x_e^{f+1} \} = x_e^f = x_e^{f+1}$. Let us now modify the flow $x^f$ into flow $\tilde{x}^f$ by sending one unit of flow along the cycle $C$. Since $C$ has a positive cost, the corresponding total cost increases to $\tilde{c}^f>c^f$. 
Now, we repeat the above reasoning by using $\tilde{x}^f$ instead of $x^f$. Notice that this procedure is finite, as we strictly increase the costs of the flows in each iteration. After the end of this procedure, there cannot exist a worsening cycle.
\end{proof}

Combining the two conditions, we obtain a complete characterization of the more-for-less paradox stated in \cref{thm:paradox}. Note that the characterization is based only on the underlying graph and costs associated with the arcs.

\begin{theorem}\label{thm:paradox}
Let $G = (V, E)$ be a directed graph with a source~$s \in V$, sink $t \in V$ and a nonnegative cost function $c\colon E \to \R^+$. 
Then, the following statements are equivalent:
\begin{enumerate}[label={(\alph*})]
    \item There exist interval capacities $\inum{u}\colon E \to \IR$ and a flow size $f$ for which the more-for-less paradox occurs.
    \item There exist interval capacities $\inum{u}\colon E \to \IR$, a flow $x^f$ of size~$f$ feasible for a scenario $u \in \inum{u}$ and an augmenting path for $x^f$ (with respect to $u$) with a negative cost.    
\end{enumerate}
\end{theorem}

\begin{proof}
The theorem follows directly from \cref{propParadoxNec,propParadoxSuf}.  
\end{proof}

It is an open question whether the condition can be weakened to the statement that the paradox occurs if and only if there is an improving path in the network, regardless of some initial flow. Here, an \emph{improving path} for a graph $G = (V,E)$ with a nonnegative cost function $c \colon E \rightarrow \R^+$ refers to a sequence of arcs $e_1=(s,v_1),\,e_2,\dots,\,e_k,\,e_{k+1}=(v_{k},t) \in E$ that form an undirected $s$--$t$ path, where the total cost of the path is negative (a forward arc $e$ contributes $c_e$ to the total cost, whereas a backward arc $e$ contributes $-c_e$). In \cref{propParadoxKn} we prove the weakened condition for complete graphs. First, we observe a characterization of improving paths with the most negative costs on complete graphs, which will be helpful for the subsequent proofs.

\begin{lemma}\label{prop:complete:back}
For a directed complete graph, the improving path with the most negative cost corresponds to a sequence of arcs $e_1=(s,v_1),\,e_2,\dots,\,e_k,\,e_{k+1}=(v_{k},t)$, where all arcs $e_2, \dots, e_k$ can be chosen as backward arcs.
\end{lemma}
\begin{proof}
    Let $P$ denote the improving path with the most negative cost and assume that there is a forward arc $e_i = (v_{i-1}, v_i) \in E$ on $P$ for some $i \in \{2, \dots, k\}$. When replacing $e_i$ with the backward arc $(v_{i}, v_{i-1}) \in E$, the total cost of the path $P$ changes by $-2\cdot c_{e_i}$. However, since all arc costs are nonnegative, using the backward arc instead of the forward arc cannot increase the cost of $P$.
\end{proof}

\begin{theorem}[Paradox for the complete graph]\label{propParadoxKn}
Let $G = (V, E)$ be the directed complete graph with a source~$s \in V$, sink $t \in V$ and a nonnegative cost function $c\colon E \to \R^+$. Then, the following statements are equivalent:
\begin{enumerate}[label={(\alph*})]
    \item There exist interval capacities $\inum{u}\colon E \to \IR$ and a flow size $f$ for which the more-for-less paradox occurs.
    \item There is an improving path from $s$ to~$t$ with a negative cost.
\end{enumerate}
\end{theorem}

\begin{proof}
\quo{(a) $\Rightarrow$ (b)}
By \cref{propParadoxNec}, there exists a flow $x^f$ feasible for a scenario $u \in \inum{u}$ and an augmenting $s$--$t$ path $P$ for $x^f$ with a negative cost. The path $P$ is also an improving path for the network.

\quo{(b) $\Rightarrow$ (a)}
Let $P$ be an improving path from $s$ to~$t$ with the most negative cost and let $f$ be the number of backward arcs on~$P$. For each backward arc $e=(i,j)$ with $i,j\not\in\{s,t\}$, set the flow of unit size along the path $s$--$j$--$i$--$t$. In total, we obtain a flow $x^f$ of size~$f$. 

We construct a scenario $u$ with the corresponding capacities $u_e = 1$ for all arcs used by the flow $x^f$, and capacities $u_e = 0$ for all remaining arcs. Moreover, we set the interval capacities $\inum{u}_e\coloneqq[0,1]$ for those arcs that are used in the above procedure (i.e., those that are involved in the flow $x^f$ or that are on the improving path~$P$); the others are set to $\inum{u}_e\coloneqq[0,0]$. 

Now, by the construction, we have a scenario $u$ of the minimum cost flow problem with a requested flow of size~$f$ and the corresponding optimal flow $x^f$. This is also the worst case scenario for a flow of size~$f$ with respect to the interval capacities $\inum{u}$, since otherwise we have a contradiction that $P$ is the most negative path with respect to the cost. Therefore, the cost of $x^f$ is the worst optimal value $c_w^f$.

For requested flow size $f+1$, there exists only one feasible flow $x^{f+1}$ (with the worst optimal value $c_w^{f+1}$), which results from $x^f$ by using the augmenting path~$P$. Since the path~$P$ has negative cost, the cost of $x^{f+1}$ is lower than the cost of $x^f$ and we have $c_w^{f+1} < c_w^{f}$, meaning that the paradox occurs.
\end{proof}

\subsection{Immune cost matrices}

Finally, we can also examine properties of the cost matrices that do not allow the more-for-less paradox for any setting of the interval capacities, i.e. matrices that are \emph{immune} against the paradox. \cref{prop:immune} shows that if we consider all such immune cost matrices (for the complete graph), the set forms a convex polyhedral cone.

\begin{proposition}\label{prop:immune}
Let $\mna{C}$ be the set of all cost matrices for the directed complete graph on $n$ nodes such that the more-for-less paradox occurs for no interval capacities. Then $\mna{C}$ is a convex polyhedral cone.
\end{proposition}

\begin{proof}
By \cref{propParadoxKn}, for given costs the paradox does not occur if and only if every path from $s$ to~$t$ is non-improving. There are finitely many paths from $s$ to~$t$ and for each of them the condition on non-improving cost has the form of a homogeneous linear inequality. 
\end{proof}

By the proof of \cref{prop:immune}, the set $\mna{C}$ admits a description by means of a system of linear inequalities. However, the number of inequalities required for the description is exponential. This necessity of such a large number of inequalities is justified by the following intractability result.

\begin{proposition}
Given costs $c$, the decision problem $c\in\mna{C}$ is strongly co-NP-hard.
\end{proposition}

\begin{proof}
To check for $c\in\mna{C}$, we just need to check if there is no improving path with a negative cost. As shown in \cref{prop:complete:back}, for the complete graph, the most negative improving path consists of a forward arc $(s,i)$, then a series of backward arcs, and finally the forward arc $(j,t)$.  Thus the most negative improving path has the cost \[ \min_{i,j\not\in\{s,t\}} \big(c_{si}-ldp(i,j)+c_{jt}\big),\] where $ldp(i,j)$ is the largest cost of the directed path from $j$ to~$i$. 

Therefore, we construct a reduction from the strongly NP-hard problem of the longest path in a graph~$G$. We set the cost of the arcs in $G$ to be~1, and then make the graph complete by adding the arcs of cost~0. Eventually, we add nodes $s$ and~$t$ and all arcs from and to $s$ and $t$ with cost~0. Now, the most negative improving path from $s$ to~$t$ corresponds to the longest path in the original graph~$G$.
\end{proof}

%%%%%%%%%%%%%%%%%%%%%%%%%%%%%%%%%%%%%%%%%%%%%%%%%%%%%%%%%%%%%%% 
\section{Conclusion}\label{sec:conclusion}
We studied the properties of the worst optimal value and the corresponding worst-case scenarios in minimum cost network flow problems with interval uncertainty in arc capacities, where each capacity may take any value within its prescribed bounds. We first established and proved fundamental convexity properties of both the set of feasible capacities and the optimal value function. We then demonstrated that computing the worst optimal value leads to a strongly NP-hard problem, and remains NP-hard even for the class of series–parallel graphs and for the class of graphs with linear number of paths.

Further, we proposed an exact method for computing the worst optimal value based on a mixed-integer linear programming formulation based on the theory of duality and complementary slackness in linear programming. For series–parallel graphs, we developed a pseudopolynomial algorithm whose complexity is polynomial in the number of nodes, arcs, and the required flow size.

We also analyzed the structure of worst-case scenarios, proving that in an extremal worst scenario, the arcs whose capacities are not set to their interval bounds form a forest. The result implies that the number of such arcs is at most $n-1$, where $n$ is the number of nodes in the network. We also constructed a class of instances for which this bound is attained, thereby proving its tightness. 

Finally, we discussed the more-for-less paradox, which arises in instances where increasing the required flow size can lead to decreasing the worst-case optimal cost. We provided a general characterization of this paradox based on augmenting paths for feasible flows, and established a stronger result for complete graphs. Moreover, we also addressed the properties of cost matrices immune to the paradox and proved that testing whether a given cost matrix is immune is a strongly co-NP-hard problem.

Future research may focus on a deeper analysis of the structural properties of worst-case scenarios. 
An interesting task is the identification of arcs that are provably included in, or excluded from, at least one or every worst-case scenario, together with sufficient or necessary conditions for such classifications. 

Furthermore, given that the problem has been shown to be NP-hard in general, it is of particular interest to investigate the computation of the worst optimal value and the corresponding worst scenarios for special classes of graphs (such as acyclic or sparse graphs and networks with limited interval uncertainty), as well as to develop efficient approximation methods for estimating the worst optimal value.

\section*{CRediT authorship contribution statement}

\textbf{Miroslav Rada}: Conceptualization,  Methodology, Writing - Original Draft, Writing - Review \& Editing, Validation. 
\textbf{Milan Hladík}: Conceptualization, Methodology, Writing - Original Draft, Writing - Review \& Editing, Supervision. 
\textbf{Elif Radová Garajová}: Conceptualization, Methodology, Writing - Original Draft, Writing - Review \& Editing, Validation.
\textbf{Francesco Carrabs}: Conceptualization, Methodology, Writing - Review \& Editing, Validation. 
\textbf{Raffaele Cerulli}: Conceptualization, Methodology, Writing - Review \& Editing, Supervision.
\textbf{Ciriaco D'Ambrosio}: Conceptualization, Methodology, Writing - Review \& Editing, Validation.

\section*{Declaration of competing interest}
The authors declare that they have no known competing financial interests or personal relationships that could have appeared to influence the work reported in this paper.

\section*{Data availability}
Not applicable.

\section*{Funding} 
M. Rada and E. Radová Garajová were supported by the Czech Science Foundation under Grant 23-07270S. M. Hladík was supported by the Czech Science Foundation under Grant 25-15714S.

%%%%%%%%%%%%%%%%%%%%%%%%%%%%%%%%%%%%%%%%%%%%%%%%%%%%%%%%%%%%%%% 
% REFERENCES
%%%%%%%%%%%%%%%%%%%%%%%%%%%%%%%%%%%%%%%%%%%%%%%%%%%%%%%%%%%%%%% 

\bibliographystyle{elsarticle-harv}
\bibliography{int_flow}

@article{Hla2018d,
 author = "Milan Hlad\'{\i}k",
 title = "The worst case finite optimal value in interval linear programming",
 journal = "Croat. Oper. Res. Rev.",
 fjournal = "Croatian Operational Research Review",
 volume = "9",
 number = "2",
 pages = "245-254",
 year = "2018",
doi="10.17535/crorr.2018.0019",
}

@book{ahuja_network_1993,
	address = {Upper Saddle River, NJ},
	edition = {1st},
	title = {Network {Flows}: {Theory}, {Algorithms}, and {Applications}},
	isbn = {978-0-13-617549-0},
	shorttitle = {Network {Flows}},
	language = {English},
	publisher = {Pearson},
	author = {Ahuja, Ravindra and Magnanti, Thomas and Orlin, James},
	month = feb,
	year = {1993},
}

@Inbook{Smith2013,
author="Smith, J. Cole
and Prince, Mike
and Geunes, Joseph",
editor="Pardalos, Panos M.
and Du, Ding-Zhu
and Graham, Ronald L.",
title="Modern Network Interdiction Problems and Algorithms",
bookTitle="Handbook of Combinatorial Optimization",
year="2013",
publisher="Springer New York",
address="New York, NY",
pages="1949--1987",
isbn="978-1-4419-7997-1",
doi="10.1007/978-1-4419-7997-1_61",
}

@inproceedings{Hoppmann:FindingMaximumMinimum:2018,
	address = {Cham},
	title = {Finding {Maximum} {Minimum} {Cost} {Flows} to {Evaluate} {Gas} {Network} {Capacities}},
	isbn = {978-3-319-89920-6},
	doi = {10.1007/978-3-319-89920-6_46},
	language = {en},
	booktitle = {Operations {Research} {Proceedings} 2017},
	publisher = {Springer International Publishing},
	author = {Hoppmann, Kai and Schwarz, Robert},
	editor = {Kliewer, Natalia and Ehmke, Jan Fabian and Borndörfer, Ralf},
	year = {2018},
	pages = {339--345},
}

@article{Ple1979,
 author = {J. Plesn{\'i}k},
 title = {The {NP}-completeness of the hamiltonian cycle problem in planar diagraphs with degree bound two},
 journal = {Information Processing Letters},
 volume = {8},
 number = {4},
 pages = {199-201},
 year = {1979},
 doi = {10.1016/0020-0190(79)90023-1},
}

@article{Hoppmann:2021:MinCostFlow,
author = {Hoppmann-Baum, Kai},
title = {On the Complexity of Computing Maximum and Minimum Min-Cost-Flows},
journal = {Networks},
volume = "79",
number = "2",
pages = {236-248},
month= "March",
year = {2022},
doi = {10.1002/net.22060paper},
}

@book{Kasperski2008,
author = {Kasperski, Adam},
year = {2008},
title = {Discrete Optimization with Interval Data - Minmax Regret and Fuzzy Approach},
series = "Studies in Fuzziness and Soft Computing",
volume = {228},
publisher = "Springer",
address = "Berlin, Heidelberg",
isbn = {978-3-540-78483-8},
doi = {10.1007/978-3-540-78484-5}
}

@article{MCFparadox,
author = {A. Gupta and M. C. Puri},
title = {``{M}ore(Same)-for-Less'' Paradox In Minimal Cost Network Flow Problem},
journal = {Optimization},
volume = {33},
number = {2},
pages = {167-177},
year  = {1995},
publisher = {Taylor & Francis},
doi = {10.1080/02331939508844073},
}

@article{CARRABS2021102492,
title = {An improved heuristic approach for the interval immune transportation problem},
journal = {Omega},
volume = {104},
pages = {102492},
year = {2021},
issn = {0305-0483},
doi = {10.1016/j.omega.2021.102492},
author = {Francesco Carrabs and Raffaele Cerulli and Ciriaco D’Ambrosio and Federico {Della Croce} and Monica Gentili}
}

@InProceedings{Singh-sensAnalysis,
author="Singh, Sanjeet
and Gupta, Pankaj
and Bhatia, Davinder",
editor="Gervasi, Osvaldo
and Gavrilova, Marina L.
and Kumar, Vipin
and Lagan{\'a}, Antonio
and Lee, Heow Pueh
and Mun, Youngsong
and Taniar, David
and Tan, Chih Jeng Kenneth",
title="On Multiparametric Sensitivity Analysis in Minimum Cost Network Flow Problem",
booktitle="Computational Science and Its Applications -- ICCSA 2005",
year="2005",
publisher="Springer",
address="Berlin, Heidelberg",
pages="1190--1202",
isbn="978-3-540-32309-9",
doi="10.1007/11424925_124"
}

@article{wollmer:1970:algorithmsfortargetingstrikes,
author = {Wollmer, Richard D.},
title = {Algorithms for Targeting Strikes in a Lines-of-Communication Network},
journal = {Operations Research},
volume = {18},
number = {3},
pages = {497-515},
year = {1970},
doi = {10.1287/opre.18.3.497}
}

@article{HASHEMI20061200,
title = {Combinatorial algorithms for the minimum interval cost flow problem},
journal = {Applied Mathematics and Computation},
volume = {175},
number = {2},
pages = {1200-1216},
year = {2006},
issn = {0096-3003},
doi = {10.1016/j.amc.2005.08.044},
author = {S. Mehdi Hashemi and Mehdi Ghatee and Ebrahim Nasrabadi}
}

@article{cenciarelli_polynomial-time_2019,
	title = {A {Polynomial}-{Time} {Algorithm} for {Detecting} the {Possibility} of {Braess} {Paradox} in {Directed} {Graphs}},
	volume = {81},
	issn = {1432-0541},
	doi = {10.1007/s00453-018-0486-6},
	number = {4},
	journal = {Algorithmica},
	author = {Cenciarelli, Pietro and Gorla, Daniele and Salvo, Ivano},
	month = apr,
	year = {2019},
	pages = {1535--1560},
}

@InProceedings{wojtcak:2018:knapsackstronglynphard,
author="Wojtczak, Dominik",
editor="Fomin, Fedor V.
and Podolskii, Vladimir V.",
title="On Strong {NP}-Completeness of Rational Problems",
booktitle="Computer Science -- Theory and Applications",
year="2018",
publisher="Springer International Publishing",
address="Cham",
pages="308--320",
isbn="978-3-319-90530-3",
doi="10.1007/978-3-319-90530-3_26"
}

@article{Busing:RobustMinimumCost:2022,
  title = {Robust Minimum Cost Flow Problem under Consistent Flow Constraints},
  author = {B{\"u}sing, Christina and Koster, Arie M. C. A. and Schmitz, Sabrina},
  year = {2022},
  month = may,
  journal = {Annals of Operations Research},
  volume = {312},
  number = {2},
  pages = {691--722},
  issn = {1572-9338},
  doi = {10.1007/s10479-021-04426-0},
  urldate = {2025-05-26},
  langid = {english}
}

@article{Chassein:ComplexityStrictRobust:2019,
  title = {Complexity of Strict Robust Integer Minimum Cost Flow Problems: {{An}} Overview and Further Results},
  shorttitle = {Complexity of Strict Robust Integer Minimum Cost Flow Problems},
  author = {Chassein, Andr{\'e} and Kinscherff, Anika},
  year = {2019},
  month = apr,
  journal = {Computers \& Operations Research},
  volume = {104},
  pages = {228--238},
  issn = {0305-0548},
  doi = {10.1016/j.cor.2018.12.021},
  urldate = {2025-05-26}
}

@article{Smith:SurveyNetworkInterdiction:2020,
  title = {A Survey of Network Interdiction Models and Algorithms},
  author = {Smith, J. Cole and Song, Yongjia},
  year = 2020,
  month = jun,
  journal = {European Journal of Operational Research},
  volume = {283},
  number = {3},
  pages = {797--811},
  issn = {0377-2217},
  doi = {10.1016/j.ejor.2019.06.024},
  urldate = {2025-10-23}
}

@book{Ford:FlowsNetworks:1962,
  title = {Flows in Networks},
  author = {Ford, L. R. and Fulkerson, D. R.},
  year = 1962,
  series = {Research Studies / {{Rand Corporation}}},
  publisher = {Princeton University Press},
  address = {Princeton, N.J},
  isbn = {978-0-691-07962-2},
  langid = {english},
  lccn = {QA265 .F6}
}

@inproceedings{Krile:ApplicationMinimumCost:2004,
  title = {Application of the Minimum Cost Flow Problem in Container Shipping},
  booktitle = {Proceedings. {{Elmar-2004}}. 46th {{International Symposium}} on {{Electronics}} in {{Marine}}},
  author = {Krile, S.},
  year = 2004,
  month = jun,
  pages = {466--471},
  issn = {1334-2630},
  url = {https://ieeexplore.ieee.org/document/1356421/references},
  urldate = {2025-11-13}
}

@article{Slump:NetworkFlowApproach:1982,
  title = {A Network Flow Approach to Reconstruction of the Left Ventricle from Two Projections},
  author = {Slump, Cornelis H and Gerbrands, Jan J},
  year = 1982,
  month = jan,
  journal = {Computer Graphics and Image Processing},
  volume = {18},
  number = {1},
  pages = {18--36},
  issn = {0146-664X},
  doi = {10.1016/0146-664X(82)90097-1},
  urldate = {2025-11-13}
}

@article{Su:MinimumcostNetworkFlow:2013,
  title = {A Minimum-Cost Network Flow Approach to Preemptive Parallel-Machine Scheduling},
  author = {Su, Ling-Huey and Cheng, T. C. E. and Chou, Fuh-Der},
  year = 2013,
  month = jan,
  journal = {Computers \& Industrial Engineering},
  volume = {64},
  number = {1},
  pages = {453--458},
  issn = {0360-8352},
  doi = {10.1016/j.cie.2012.04.020},
  urldate = {2025-11-13}
}

@article{Estes:FacetsStochasticNetwork:2020,
  title = {Facets of the {{Stochastic Network Flow Problem}}},
  author = {Estes, Alexander S. and Ball, Michael O.},
  year = 2020,
  month = jan,
  journal = {SIAM Journal on Optimization},
  volume = {30},
  number = {3},
  pages = {2355--2378},
  publisher = {{Society for Industrial and Applied Mathematics}},
  issn = {1052-6234},
  doi = {10.1137/19M1286049},
  urldate = {2025-11-13}
}

@article{Ghatee:GeneralizedMinimalCost:2008,
  title = {Generalized Minimal Cost Flow Problem in Fuzzy Nature: {{An}} Application in Bus Network Planning Problem},
  shorttitle = {Generalized Minimal Cost Flow Problem in Fuzzy Nature},
  author = {Ghatee, Mehdi and Hashemi, S. Mehdi},
  year = 2008,
  month = dec,
  journal = {Applied Mathematical Modelling},
  volume = {32},
  number = {12},
  pages = {2490--2508},
  issn = {0307-904X},
  doi = {10.1016/j.apm.2007.09.030},
  urldate = {2025-11-13}
}

@inproceedings{Mao:RobustDiscreteOptimization:2009,
  title = {Robust {{Discrete Optimization}} for the {{Minimum Cost Flow Problem}}},
  booktitle = {2009 {{International Workshop}} on {{Intelligent Systems}} and {{Applications}}},
  author = {Mao, Rui and Zhu, Jinfu},
  year = 2009,
  month = may,
  pages = {1--4},
  doi = {10.1109/IWISA.2009.5073102},
  urldate = {2025-11-13}
}

@article{DAmbrosio:OptimalValueRange:2020,
  title = {The Optimal Value Range Problem for the {{Interval}} (Immune) {{Transportation Problem}}},
  author = {D'Ambrosio, C. and Gentili, M. and Cerulli, R.},
  year = 2020,
  month = sep,
  journal = {Omega},
  volume = {95},
  pages = {102059},
  issn = {0305-0483},
  doi = {10.1016/j.omega.2019.04.002},
  urldate = {2023-05-23},
  langid = {english}
}

@inproceedings{Garajova:QuasiextremeReductionInterval:2024,
  title = {A {{Quasi-extreme Reduction}} for~{{Interval Transportation Problems}}},
  booktitle = {Dynamics of {{Information Systems}}},
  author = {Garajov{\'a}, Elif and Rada, Miroslav},
  editor = {Moosaei, Hossein and Hlad{\'i}k, Milan and Pardalos, Panos M.},
  year = 2024,
  pages = {83--92},
  publisher = {Springer Nature Switzerland},
  address = {Cham},
  doi = {10.1007/978-3-031-50320-7_6},
  isbn = {978-3-031-50320-7},
  langid = {english}
}

@article{Garajova:IntervalTransportationProblem:2023,
  title = {Interval Transportation Problem: Feasibility, Optimality and the Worst Optimal Value},
  shorttitle = {Interval Transportation Problem},
  author = {Garajov{\'a}, Elif and Rada, Miroslav},
  year = 2023,
  month = sep,
  journal = {Central European Journal of Operations Research},
  volume = {31},
  number = {3},
  pages = {769--790},
  issn = {1613-9178},
  doi = {10.1007/s10100-023-00841-9},
  urldate = {2023-06-30},
  langid = {english}
}

\end{document}